\documentclass[12pt,reqno]{amsart}
\usepackage{ae}
\usepackage[all]{xy}
\usepackage{graphicx,amsfonts,amssymb,amsmath}
\usepackage{pifont}
\usepackage{tikz-cd}
\usepackage{enumitem}

\usepackage{extpfeil} 

\usepackage[backref]{hyperref}
\hypersetup{colorlinks,linkcolor=blue,citecolor=red,}
\usepackage[alphabetic,backrefs,lite]{amsrefs}

\usepackage{geometry}
\usepackage[OT2,T1]{fontenc}
\DeclareSymbolFont{cyrletters}{OT2}{wncyr}{m}{n}
\DeclareMathSymbol{\sha}{\mathalpha}{cyrletters}{"58}

\input cyracc.def

 \newtheorem{thm}{Theorem}[section]
  \newtheorem*{ques1}{Question 1}
\newtheorem*{ques2}{Question 2}

 \newtheorem{cor}[thm]{Corollary}
 \newtheorem{lem}[thm]{Lemma}
 
 \newtheorem{prop}[thm]{Proposition}
 \theoremstyle{definition}
 \newtheorem{defn}[thm]{Definition}
 \theoremstyle{remark}
 
 \theoremstyle{remark}
 \newtheorem{rem}[thm]{Remark}

 \newcommand{\To}{\longrightarrow}

 \newcommand{\N}{\textup{N}}
 \newcommand{\F}{\mathbb{F}}

 \newcommand{\coker}{\textup{Coker}}
 \newcommand{\im}{\textup{Im}}
 \renewcommand{\ker}{\textup{Ker}}
 \newcommand{\Pic}{\textup{Pic}}

 \newcommand{\Br}{\textup{Br}}

 \newcommand{\Gal}{\textup{Gal}}

  \newcommand{\Spec}{\textup{Spec}}

 \renewcommand{\P}{\mathbb{P}}

 \newcommand{\A}{\textbf{A}}
 \newcommand{\Q}{\mathbb{Q}}
 
  \newcommand{\C}{\mathbb{C}}
 \newcommand{\Z}{\mathbb{Z}}
 \newcommand{\G}{\mathbb{G}}

 \newcommand{\br}{\mathrm{Br}}
 \newcommand{\inv}{\mathrm{inv}}
 \newcommand{\HH}{\mathrm{H}}
 \newcommand{\PP}{\mathbb{P}}
 \newcommand{\Af}{\mathbb{A}}
 \newcommand{\Ad}{\mathbf{A}}
 
 \newcommand{\UU}{\mathcal{U}}

\numberwithin{equation}{section}

\begin{document}

\title[]
 {Varieties with prescribed finite unramified Brauer groups and subgroups precisely obstructing the Hasse principle}

\author{Yongqi Liang}
\author{Yufan Liu}

\address{Yongqi LIANG
\newline University of Scinece and Technology of China,
\newline School of Mathematical Sciences,
\newline 96 Jinzhai Road,
\newline  230026 Hefei, Anhui, China
 }

\email{yqliang@ustc.edu.cn}

\address{Yufan Liu
\newline University of Scinece and Technology of China,
\newline School of Mathematical Sciences,
\newline 96 Jinzhai Road,
\newline  230026 Hefei, Anhui, China
 }

\email{liuyufan@mail.ustc.edu.cn}

\keywords{normic bundles, Brauer group}
\thanks{\textit{MSC 2020} : 14G12, 14G05, 14F22}

\begin{abstract}
On varieties defined over number fields, we consider obstructions to the Hasse principle
given by subgroups of their Brauer groups.

Given an arbitrary pair of non-zero finite abelian groups $B_0\subset B$, we prove the existence of a variety $X$ such that its unramified Brauer group is isomorphic to $B$ and moreover $B_0$ is the smallest subgroup of $B$ that obstructs the Hasse principle. The concerned varieties   are normic bundles over the projective line.
\end{abstract}

\maketitle

\section{Introduction}

For algebraic varieties $X$ defined over number fields $k$, a fundamental problem in arithmetic geometry is to determine the existence of $k$-rational points on $X$. A necessary condition of such an existence is the existence of $k_v$-rational points for all places $v\in\Omega_k$ of $k$. If, for a certain class of varieties, this is also a sufficient condition, we say that the Hasse principle holds. It turns out that many varieties violate the Hasse principle. To explain these violations, in 1970s Manin \cite{Manin} introduced a (nowadays called Brauer\textendash Manin) pairing
$$\langle ~,~\rangle:X(\A_k)\times\Br(X)\to\Q/\Z$$
such that the diagonal image of $X(k)\to X(\A_k)$ is contained in the Brauer\textendash Manin subset
$$X(\A_k)^B=\{(x_v)_{v\in\Omega_k}\mid ~\forall b\in B, \langle (x_v),b\rangle=0\}$$
with respect to a subgroup $B\subset\Br(X)$ of the Brauer group. When $\varnothing=X(\A_k)^B\subsetneq X(\A_k)$, we say that the subgroup $B$ obstructs the Hasse principle. Colliot-Th\'el\`ene conjectured that violations of the Hasse principle for (geometrically) rationally connected varieties are due to the obstruction given by the Brauer group, cf. \cite{CTSkbook}*{Conjecture 14.1.2}.

In principle, the study of the Brauer\textendash Manin obstruction is divided into two steps. The first step is to compute the unramified  Brauer group of a geometrically connected smooth variety (or equivalently the Brauer group of its smooth compactifications). The second step is to determine whether an element in (a subgroup of) the Brauer group  obstructs the Hasse principle.

Concerning the first step, to compute the Brauer group is difficult in general, even though the Brauer group is known to be finite (modulo the constant part) for any rationally connected variety. The following question arises naturally.
\begin{ques1}
Given an arbitrary finite abelian group $B$, does there exist a rationally connected variety $X$ such that $\coker(\Br(k)\to\Br(X))\simeq B$?
\end{ques1}
In the literature, for most varieties with computable Brauer groups, one usually terminates at either a $p$-group for a certain prime number $p$ or a finite abelian group of small order. To the best of our knowledge, this question has not yet been seriously considered. As a side note,  T. J. Ford had an affirmative answer to the variant of this question in the class of  affine varieties defined over the field $\C$ of complex numbers, cf.  \cite{Ford}.

Concerning the second step, generally speaking, people want to understand which specific element of the Brauer group takes effect in the Brauer\textendash Manin pairing. When $X$ is proper, as the set $X(\A_k)$ of its ad\'elic points  is compact, the local constancy of the Brauer\textendash Manin pairing implies that $X(\A_k)^B=\varnothing$ for a finite subgroup $B\subset\Br(X)$ provided that $X(\A_k)^{\Br(X)}=\varnothing$. For most situations presented in the literature, it suffices to take $B$ to be generated by one single element. In a recent paper \cite{Berg-etal} of J. Berg et al., the authors constructed  conic bundles over the projective line, for which the Brauer\textendash Manin obstruction to the Hasse principle requires arbitrarily many Brauer classes. We refer to the introduction of their paper for more details and a brief history of this research direction. In the only examples
\cite{CT-SD}*{Proposition 10.1}\cite{KreschTschinkel}*{Example 7}\cite{Corn}*{Example 9.4}\cite{Berg-etal} requiring more than one Brauer class, the required Brauer classes are of order $2$. It is therefore natural to ask the following question, where the special case with $B_0$ any finite elementary abelian $2$-group has already been solved in \cite{Berg-etal}.
\begin{ques2}
Given an arbitrary non-zero finite abelian group $B_0$, does there exist a proper smooth geometrically connected variety $X$ violating the Hasse principle, such that its Brauer group $\Br(X)$ contains $B_0$ as a subgroup which is  required to obstruct the Hasse principle, i.e. $B_0$ obstructs the Hasse principle while no proper subgroup of $B_0$ obstructs?
\end{ques2}

In the present paper, we simultaneously answer to Questions 1 and 2. To be more precise, we state our main result as follows.

\begin{thm}[Theorem \ref{thm2}]
Let $B$ be a finite abelian group of exponent dividing $n$ and $0\neq B_0\subset B$ be a non-zero subgroup. Let $k$ be a number field containing a primitive $n$-th root of unity. Then there exists a smooth rationally connected variety $X$ defined over $k$ satisfying the following conditions.
\begin{enumerate}
\item There is a monomorphism $B\to\Br(X)$ inducing an isomorphism of groups $B\simeq\coker(\Br(k)\to\Br(X))$.
\item For any subgroup $B'\subset B$, the Brauer\textendash Manin set $X(\A_k)^{B'}=\varnothing$ if and only if $B_0\subset B'$. In particular, the variety $X$ violates the Hasse principle, and the obstruction to the Hasse principle is precisely given by the subgroup $B_0$.
\end{enumerate}
\end{thm}

The variety concerned in the theorem is a smooth compactification of the subvariety in $\mathbb{A}^{n+1}$ defined over $k$ by a normic equation
\begin{equation}\tag{$\star$}\label{norm equation}
\N_{K/k}(\mathbf{z})=P(x),
\end{equation}
where $K$ is a degree $n$ cyclic Galois extension of $k$ and $P\in k[x]$ is a separable polynomial of degree divisible by $n$. In Sections \ref{section abelian} and \ref{section cyclic}, we compute the unramified Brauer group (Theorem \ref{cyclic case}) of the equation \eqref{norm equation}  by applying the existence of a proper smooth model constructed by A. V\'arilly-Alvarado and B. Viray in \cite{VAV-model}. It turns out that such unramified Brauer groups, being determined by the factorization of $P$ over $K$, can already recover all finite abelian groups. The conclusion (1) in our theorem is therefore obtained.
When $n=2$, similar computations appeared in the book \cite{Skbook}*{Proposition 7.1.2} by A. Skorobogatov.
When $n=p$ is a prime number, these unramified Brauer groups have been computed by A. V\'arilly-Alvarado and B. Viray
in \cite{VAV-chatletp-fold}*{Theorem 3.2}.
In this case, we recover and generalise their result by removing the technical constraints on the splitting fields of irreducible factors of $P$.
In a preprint of D. Wei \cite{Wei2022+}, he constructed an ``almost proper" model to compute the unramified Brauer group for even more general normic equations. For more information on this topic, we refer to \cite{Wei2022+} and the book written by J.-L. Colliot-Th\'el\`ene and A. Skorobogatov \cite{CTSkbook}*{Remark 11.1.12}.

Replacing the conic bundle in \cite{Berg-etal} by the normic equation above, with a similar outline of proofs we are able to generalise its main result to the  conclusion (2) in our theorem. More details are given in Section \ref{section maintheorem2}.
We should also remark that the conclusion (2) is even an answer to a stronger version of Question 2. Indeed, it states that $B_0$ \textit{precisely obstructs} the Hasse principle. In other words, the subgroup $B_0$ is not only \textit{a minimal subgroup} but also \textit{the smallest subgroup} of $\Br(X)$ such that $X(\A_k)^{B_0}=\varnothing$. It follows a fortiori that $B_0$ \textit{completely captures} (cf.  \cite{Viray23}*{Definition 31(2)}) the Brauer\textendash Manin obstruction in the sense of B. Creutz and B. Viray.

\section{Notation and preliminaries}

\subsection{Notation and conventions}\label{notation}

In the present paper, we denote by $k$ a field of characteristic $0$.  We fix an algebraic closure $\overline{k}$ of $k$. The absolute Galois group of $k$ is denoted by $\Gamma_k=\Gal(\overline{k}/k)$. All finite extension of $k$ will be considered as a subfield of $\overline{k}$.

When $k$ is a number field, we denote by $\Omega_k $ the set of all places of $k$.
For any $v\in \Omega_k$, let $k_v$ be the corresponding completion. When $v$ is non-archimedean, we denote by $\F_v$ its residue field.

A variety over $k$ is understood as a separated scheme of finite type over $k$.
For any variety $X$ defined over $k$, we denote by  $\br(X)=\HH^2_\text{\'et}(X,\G_m)$ its cohomological  Brauer group whose  subgroup  consisting of constant classes is denoted by $\br_0(X)=\im(\br(k)\rightarrow\br(X))$.
We write $\overline{\br}(X)$ (respectively $\overline{\br}_\textup{vert}(X)$) for the quotient $\br(X)/\br_0(X)$ (respectively $\br(X)_\textup{vert}/\br_0(X)$) for simplicity.

\subsection{Preliminaries}

Before the beginning of the proof of our main theorem, we present several preliminaries in this section. Although some of them are probably well-known, we still give some details for the self-containedness of the paper.

\subsubsection{Galois theory and norm}
\begin{lem}\label{galoistheory}
    Let  $K$ be a finite Galois extension of $k$. Then, for any finite extension $l$ of $k$, the compositum $Kl$ is a Galois extension of $l$ with Galois group $\Gal(Kl/l)\simeq\Gal(K/K\cap l)$ a subgroup of $\Gal(K/k)$.
\end{lem}

\begin{proof}
    This is classic, we sketch its proof as follows.
    If $K$ is the splitting field of a certain separable polynomial $f\in k[x]$ over $k$, then $Kl$ is the splitting field of $f$ over $l$, which is therefore an Galois extension. The monomorphism $\sigma\mapsto\sigma_{|K}$ identifies $\Gal(Kl/l)$ as a subgroup of $\Gal(K/k)$. The fixed subfield of $\Gal(Kl/l)$ is exactly $K\cap l$. We refer to \cite[Theorem 5.5]{GTM167} for more details.
\end{proof}

\begin{prop}\label{polynorm}
    Let $l\subset\overline{k}$ be a rupture field of a monic irreducible polynomial $P\in k[x]$. Let  $K$ be a finite Galois extension of $k$. 
    If we denote $l'=K\cap l$, then $P$ is a norm for the extension $l'(x)/k(x)$.
\end{prop}

\begin{proof}

    We write $l=k(\alpha)$ with $\alpha\in\overline{k}$ a root of $P$.
    Let $P=\prod_{j=1}^{r}P_j$ be the monic irreducible factorization of $P$ over the extension $K$.  Without loss of generality,  we may assume that $\alpha$ is a root of  $P_1$. Then $P_1$ is the minimal polynomial of $\alpha$ over $K$.

    The  group $\Gal(K/k)$ acts naturally on the set of factors $\{P_1,\ldots,P_r\}\subset K[x]$. As the product of all elements in each orbit is stable under the action of $\Gal(K/k)$, it is a polynomial in $k[x]$. Since such a product divides $P$, it must equal  $P$ by monic irreducibility. In other words, the action is transitive. It follows that all the factors $P_j$ are of the same degree $\frac{\deg(P)}{r}$.

    According to Lemma \ref{galoistheory}, the subgroup $\Gal(K/K')\simeq\Gal(Kl/l)$ of $\Gal(K/k)$ fixes $\alpha\in l$, then it also fixes the minimal polynomial $P_1$ of $\alpha$ over $K$. It follows that $\Gal(K/K')$ is a subgroup of the stabilizer $G_1$ of $P_1$. Indeed they are the same subgroup of $\Gal(K/k)$ since
\begin{align*}
|\Gal(K/K')|=|\Gal(Kl/l)|=[Kl:l]&=\frac{[Kl:K][K:k]}{[l:k]}\\
&=\frac{\mathrm{deg}(P_1)[K:k]}{\mathrm{deg}(P)}=\frac{[K:k]}{r}=\frac{|\Gal(K/k)|}{r}=|G_1|.
\end{align*}
    For each $1\leq j\leq r$, choose $\sigma_j\in\Gal(K/k)$ such that $\sigma_j P_1=P_j$, then $\Gal(K/k)=\bigsqcup_{j=1}^r \sigma_j \Gal(K/K')$.
    So we can write  $\mathrm{Hom}_k(K',\overline{k})=\{{\sigma_1}_{|K'},\ldots,{\sigma_r}_{|K'}\}$ and we have
    \[
        \N_{K'(x)/k(x)}(P_1)=\prod_{j=1}^r \sigma_j P_1=\prod_{j=1}^rP_j=P.
    \]
\end{proof}

Let $K/k$ be a Galois extension of degree $n$, we define a \emph{normic polynomial} $\N_{K/k}(\mathbf{z})\in k[\mathbf{z}]$ in one $K$-variable $\mathbf{z}$ as follows. If we choose a $k$-linear basis $e=(e_1,\ldots,e_n)$ of $K$ and write $\mathbf{z}=\sum_{i=1}^ne_iz_i$, then $\N_{K/k}(\mathbf{z})$ is defined to be the homogeneous polynomial
$$\N_{K/k}(e_1z_1+\cdots+e_nz_n)=\prod_{\sigma\in\Gal(K/k)}\sigma(e_1z_1+\cdots+e_nz_n)\in k[z_1,\ldots,z_n]$$
in $n$ variables with coefficients in $k$ according to Galois theory. A different choice of the basis $e$ results by a $k$-linear change of variables, while it does not affect the factorisation of such a polynomial.  It follows that the isomorphic class of the normic variety over $k$ that we are going to define using the normic polynomial in Section \ref{normicvarietysubsection} is independent of the choice of the basis.

We study the factorisation of $\N_{K/k}$ over finite extensions of $k$.

\begin{lem}\label{norm}
    Let $l$ be a finite extension of $k$. Then the irreducible factors of $\N_{K/k}$ over $l$ are
    \[
        \N_{l,\tau}=\prod_{\sigma\in\Gal(Kl/l)}\sigma\tau(\mathbf{z})=\prod_{\sigma\in\Gal(Kl/l)}\sigma\tau(e_1z_1+\cdots +e_nz_n)
    \]
    with $\tau\in\Gal(K/k)$ such that $\Gal(Kl/l)\tau$ ranges over the set  $\Gal(Kl/l)\backslash\Gal(K/k)$ of right cosets.

    In particular, the polynomial $\N_{K/k}$ is irreducible over $k$.
\end{lem}

\begin{proof}
    Consider the free action of the subgroup $\Gal(Kl/l)\subset\Gal(K/k)$ on the set of linear factors of $\N_{K/k}$ over $K$
    \[
        \{\tau(\mathbf{z})=\tau(e_1z_1+\cdots +e_nz_n)\mid\tau\in\Gal(K/k)\}.
    \]
    It is a disjoint union of orbits. By Galois theory, the product $$\N_{l,\tau}=\prod_{\sigma\in\Gal(Kl/l)}\sigma\tau(e_1z_1+\cdots +e_nz_n)$$ of elements in a single orbit is a polynomial in $l[z_1,\dots,z_n]$.

    Let $\N\in l[x]$ be an irreducible factor of $\N_{K/k}$. By definition of the normic polynomial, there exists a certain $\tau\in\Gal(K/k)$ such that $\tau(e_1z_1+\cdots+e_nz_n)$ divides $\N$ in $Kl[z_1,\dots,z_n]$. Since $\sigma\tau(e_1z_1+\cdots+e_nz_n)$ divides $\sigma\N=\N$ for any $\sigma\in\Gal(Kl/l)$, we find that $\N_{l,\tau}$ divides $\N$ in $l[z_1,\dots,z_n]$. Hence $\N=\N_{l,\tau}$ by irreducibility.
\end{proof}

\begin{rem}\label{f-factorization}
Let $f\in k[x]$ be a monic irreducible polynomial such that $K=k[x]/(f)$. We can see that the factorisations over $l$ of $f$ and $\N_{K/k}$ determine each other as follows. As $K/k$ is Galois, the roots $\alpha_1,\ldots,\alpha_n$ of $f$ lie in $K$. Fixing a root $\alpha=\alpha_i$ and replacing the set $\{\tau(e_1z_1+\cdots+e_nz_n)\mid\tau\in\Gal(K/k)\}$ by $\{\tau(x-\alpha)\mid\tau\in\Gal(K/k)\}$, we obtain the factorisation of $f$ in $l[x]$ exactly in the same way. In particular, the irreducible factors of $f$ over $l$ are all of the same degree $[Kl:l]$.
\end{rem}

\subsubsection{Properties of abelian groups}

The following properties on abelian groups are used in the proof of our main result.

\begin{lem}\label{cyclicexponent}
    Let $A=\bigoplus_{i=1}^m\Z/n_i\Z$ be a finite abelian group. Then the order of the element $\mathbf{1}_m=(\overline{1} , \overline{1}, \ldots, \overline{1})$ is equal to the exponent of $A$.
\end{lem}

\begin{proof}
    According to the structure theorem of finitely generated abelian groups, the group $A$ is isomorphic to $\bigoplus_{i=1}^{m'}\Z/n'_i\Z$ with a family of integers $n'_i\mid n'_{i+1}$ for all $i$. Moreover, under this isomorphism $\mathbf{1}_m$ corresponds to $\mathbf{1}_{m'}=(\overline{1} , \overline{1}, \ldots, \overline{1})\in\bigoplus_{i=1}^{m'}\Z/n'_i\Z$ whose order is obviously equal to the exponent of the group.
\end{proof}

\begin{lem}\label{group structure}
    For  $1\leq i\leq m$, let $\varphi_i:A\rightarrow A_i$ be epimorphisms of finite abelian groups.
    Then there exists an element  $a\in A$ such that the order of $(\varphi_1(a),\ldots,\varphi_m(a))$ equals  the exponent of  $\bigoplus_{i=1}^m A_i$.
\end{lem}

\begin{proof}
    Factorise the exponent $\exp(\bigoplus_{i=1}^m A_i)=\prod p^{r}$ as a product of distinct primes with positive integers $r$. If, for each $p$, we can find $a_p\in A$ such that $(\varphi_1(a_p),\ldots,\varphi_n(a_p))$ is an element of order $p^{r}$, then we conclude by taking $a=\sum a_p$.

    The structure theorem of finitely generated abelian groups applied to these $A_i$'s implies that there exists at least one $i$ such that the $p$-adic valuation $\mathrm{ord}_{p}(\exp(A_i))=r$ and that there exists an element $a_i\in A_i$ of order $p^r$. Take $a_p$ to be a lifting of $a_i$ to the $p$-primary part of $A$ via $\varphi_i$.
\end{proof}

\begin{lem}\label{cyclicgroup}
    Let $\varphi:A\rightarrow A_3$ be a homomorphism of abelian groups whose kernel is a finite cyclic group.
    Let $A_1$ and $A_2$ be subgroups of $A$ such that $A_1\cap A_2=\{0\}$. Then $\ker(A_1\rightarrow A_3/\varphi(A_2))$ is a finite cyclic group.
\end{lem}

\begin{proof}
    First of all,  we have $\ker(A_1\rightarrow A_3/\varphi(A_2))=A_1\cap (A_2+\ker(\varphi))$. We claim that it is a subquotient of $\ker(\varphi)$, therefore it is also a finite cyclic group. To see this, we consider the subgroup $A_0=\ker(\varphi)\cap(A_1+A_2)$ of $\ker(\varphi)$. As $A_1\cap A_2=\{0\}$, any element $a_0\in A_0$ can be written uniquely as $a_0=a_1+a_2$ with $a_1\in A_1$ and $a_2\in A_2$. Then
    \begin{gather*}
        A_0\rightarrow A_1\cap (A_2+\ker(\varphi))\\
        a_0\mapsto a_1
    \end{gather*}
    defines an epimorphism of abelian groups.
\end{proof}

\subsubsection{An irreducibility criterion}
We recall the following irreducibility criterion that we are going to use frequently. The original version was stated only for a completion $L=k_v$ of $k$, but it is still valid for an arbitrary extension $L$ of $k$ without any modification of the proof.

\begin{lem}[\cite{Berg-etal}*{Lemma 4.4}]\label{irreducibility}
    Let $k$ be a field and $g\in k[x]$ be an irreducible polynomial. Let $L/k$ be a field extension containing a root $\theta$ of $g$.
    If $h\in k[x]$ is a polynomial such that $h-\theta$ is irreducible over $L$, then $g\circ h$ is irreducible over $k$.
\end{lem}

\subsubsection{Local invariants of Brauer classes}\label{local-inv}
According to the theory of local fields, when an element of the Brauer group is given by a cyclic algebra, its local invariant is easy to compute.

Let $n>0$ be an integer and $k$ be a number field containing a primitive $n$-th root of unity $\zeta=\zeta_n$.
Let $k_v$ be the local field associated to a non-archimedean place $v\in \Omega_k$. Denote by $\F_v$ its residue field. If the characteristic $\textup{char}(\F_v)$ does not divide $n$, then the image  $\bar{\zeta}$ of $\zeta$ in $\F_v$ is still a primitive $n$-root of unity. In particular, $n$ divides $|\F_v|-1$. Let
\begin{equation*}
\log:\mu_n(\F_v)\to \frac{1}{n}\Z/\Z\subset\Q/\Z
\end{equation*}
be the isomorphism mapping $\bar{\zeta}$ to $\frac{1}{n}~\textup{mod}~\Z$.

For $a\in k^*$ such that $x^n-a$ is irreducible over $k$, we fix a root $\sqrt[n]{a}\in\overline{k}$. Then $K=k(\sqrt[n]{a})$ is a cyclic extension of $k$ of Galois group $G=\Gal(K/k)\simeq\Z/n\Z$, which is generated by an element $\sigma:\sqrt[n]{a}\mapsto\zeta\sqrt[n]{a}$. Let $\chi_a:G\to\Q/\Z$ be the character mapping $\sigma$ to $\frac{1}{n}~\textup{mod}~\Z$, then the group $\widehat{G}$ of characters is generated  by $\chi_a$.

For an element $b\in k_v^*$, let $(a,b)_\zeta\in\br(k_v)$ be the class of the cyclic $k_v$-algebra  generated by symbols $x,y$ subject to the relations $x^n=a,y^n=b,xy=\zeta yx$. According to \cite[Proposition 2.5.2, Corollary 2.5.5, Proposition 4.7.3]{GilleSzamuely}, this identifies with the class $\left(\chi_a,b\right)\in\br(k_v)$ defined in terms of cup products or other equivalent form which will appear in Section \ref{section cyclic}.

\begin{prop}[\cite{serre13}*{Chapter XIV, Proposition 8 and Corollary}]\label{inv}
Let $k$ be a number field containing a primitive $n$-th root of unity and $v$ a place of $k$ such that its residue characteristic does not divide $n$.

Define $\overline{c}\in\F_v^*$ to be the reduction of the $v$-adic unit $\displaystyle c=(-1)^{\alpha\beta}\frac{a^\beta}{b^\alpha}$ where $\alpha=\mathrm{ord}_v(a)$ and $\beta=\mathrm{ord}_v(b)$.
Then
$$\inv_v(\left(\chi_a,b\right))=\log(\overline{c}^{\frac{|\F_v|-1}{n}}).$$
\end{prop}

\subsubsection{Endomorphisms of  $\P^1$} 
As an important argument of the proof of their Theorem 7.1 \cite[Section 7.2]{Berg-etal}, the authors apply weak approximation to obtain the existence of a global polynomial (of sufficiently divisible degree) satisfying certain local conditions. We summarise it as the following black box.

\begin{prop}\label{p1p1}
    Let $k$ be a number field and $v_0,\ldots,v_m,w_1,\ldots,w_s,w'_1,\dots,w'_{s'}\in \Omega_k$ be distinct non-archimedean places of $k$.
    For each $0\leq i\leq m$, let $\theta_i$ be an element of $k_{v_i}$. 
    For each $1\leq t\leq s$, let $(\UU_{w_t,\lambda})_{\lambda\in\Lambda_t}$ be a finite family of non-empty open subsets of $\P^1(k_{w_t})$ whose union $\UU_{w_t}=\bigcup_{\lambda\in\Lambda_t}\UU_{w_t,\lambda}$ contains $\infty$. For each $1\leq t'\leq s'$, let $\UU_{w'_{t'}}$ be a non-empty open subset of $\PP^1(k_{w'_{t'}})$.
    
    Then there exists a polynomial $h\in k[x]$ satisfying the following conditions:
    \begin{enumerate}
        \item[\textup{(h1)}] for each $0\leq i\leq m$, the polynomial $h-\theta_i\in k_{v_i}[x]$ is irreducible;
        \item[\textup{(h2)}] for each $1\leq t\leq s$, the image $h(\P^1(k_{w_t}))\subset \UU_{w_t}$ and it meets $\UU_{w_t,\lambda}$ for each $\lambda\in\Lambda_t$,
        \item[\textup{(h3)}] for each $1\leq t'\leq s'$, the image $h(\P^1(k_{w'_{t'}}))$ meets $\UU_{w'_{t'}}$.
    \end{enumerate}
\end{prop}

\section{Unramified Brauer groups of  abelian normic equations}\label{section abelian}

In this section, we will compute the unramified Brauer group of a class of abelian normic bundles. Results in this section are conditional on the existence of a certain model of the normic equation \eqref{norm equation}.

\subsection{Normic varieties and normic bundles}\label{normicvarietysubsection}

Let $K/k$ be Galois extension of degree $n$. For $\theta\in k$, consider the affine variety $X^\circ_{K/k,\theta}\subset\mathbb{A}^{n}$ defined over $k$ by the equation
$$\N_{K/k}(\mathbf{z})=\theta$$ in a $K$-variable  $\mathbf{z}$.  Write  $K=k(\alpha)\simeq k[x]/(f)$ with $\alpha\in\overline{k}$ a root of a monic irreducible polynomial $f\in k[x]$. As $K/k$ is Galois, the roots of $f$ are $\sigma(\alpha)$ with $\sigma\in \Gal(K/k)$. Fix the $k$-linear basis $e=(1,\alpha,\ldots,\alpha^{n-1})$ of $K$, the equation is written as
$$\N_{K/k}(z_1+\alpha z_2+\cdots+\alpha^{n-1}z_n)=\prod_{\sigma\in\Gal(K/k)}\sigma(z_1+\alpha z_2+\cdots+\alpha^{n-1}z_n)=\theta.$$

First of all, we consider the non-degenerate case  that $\theta\neq0$.
For $\sigma\in\Gal(K/k)$ and $0\leq i\leq n-1$, the $n\times n$ matrix $(\sigma(\alpha)^i)_{\sigma,i}$ is invertible. It follows that $X^\circ_{K/k,\theta}$ is isomorphic over $\overline{k}$ to the affine variety defined by
$$z'_1z'_2\cdots z'_n=\theta\neq0$$ which is smooth integral and rational. We refer to a smooth projective model $X_{K/k,\theta}$ of $X^\circ_{K/k,\theta}$ a \emph{normic variety}.

When $\theta=0$ and $K\neq k$, the $k$-variety $X^\circ_{K/k,0}$ is not smooth but reduced. Indeed, geometrically, it is a union of $n$ hyperplanes  intersecting transversally at one central point. According to Lemma \ref{norm} and Lemma \ref{galoistheory}, over a finite extension $l/k$, the variety $X^\circ_{K/k,0}\otimes_kl$ decomposes into $[K\cap l:k]$ irreducible components each of multiplicity $1$. Each component $Y_{l,\tau}$ is defined by the equation
$$\N_{l,\tau}=\prod_{\sigma\in\Gal(Kl/l)}\sigma\tau(z_1+\alpha z_2\cdots +\alpha^{n-1}z_n)=0$$
for a certain $\tau\in\Gal(K/k)$.

\begin{prop}\label{algebraicclosureinfunctionfield}
The algebraic closure of $l$ in the function field $l(Y_{l,\tau})$ of $Y_{l,\tau}$ is the compositum $Kl$.
\end{prop}

\begin{proof}
To simplify the notation, we denote $Y_{l,\tau}$ by $V$ and denote $Kl=l(\alpha)$ by $L$. We are going to compute $L\otimes_ll(V)$ in two ways.

On one hand, the $L$-variety $V\otimes_lL$ decomposes as a union $\bigcup_{\sigma\in\Gal(L/l)}W_\sigma$ of hyperplanes $W_\sigma$ defined by
$\sigma\tau(z_1+\alpha z_2\cdots +\alpha^{n-1}z_n)=0$. In terms of generic points, we obtain that $$L\otimes_ll(V)=\prod_{\sigma\in\Gal(L/l)}L(W_\sigma)$$ as a finite \'etale $l(V)$-algebra.

On the other hand, let $g\in l[x]$ be the monic minimal polynomial of $\alpha$ over $l$, whose  degree is equal to $d=|\Gal(L/l)|$. The polynomial $g$ factorises into a product $\prod_sg_s$ of at most $d$ monic irreducible polynomials over the extension $l(V)$ of $l$. Then $$L\otimes_ll(V)=l[x]/(g)\otimes_ll(V)=\prod_sl(V)[x]/(g_s).$$

The comparison of these two expression tells us  that $g$ splits completely over $l(V)$ and the extension $L(W_\sigma)/l(V)$ is trivial, in particular  $L\subset l(V)$.
As $W_\sigma$ is geometrically integral, the field $L$ is algebraically closed in $L(W_\sigma)=l(V)$. Therefore the algebraic closure of $l$ in $l(V)$ is $L$.
\end{proof}

Now we turn to the discussion of normic bundles over $\P^1$.

Let $P\in k[x]$ be a separable polynomial of degree divisible by $n=[K:k]$.
Factorise
\begin{equation}\label{Pfactorisation}
P=c\prod_{i=1}^m P_i
\end{equation}
as a product of distinct monic irreducible polynomials of degree $d_i$ with $c\in k^*$ the leading coefficient of $P$.
By abuse of notation, we also denote by $P_i$ the closed point of $\Af^1$ defined by the polynomial equation $P_i(x)=0$.

Let $X^\circ_{K/k,P}\subset\Af^{n+1}$ be the hypersurface defined over $k$ by the normic equation
\begin{equation}\tag{\ref{norm equation}}
    \N_{K/k}(\mathbf{z})=P(x).
\end{equation}
As $P$ has no multiple roots in $\overline{k}$, this hypersurface is smooth over $k$ by Jacobian criterion.

\begin{defn}
    We say that a smooth proper compactification $X=X_{K/k,P}$ of $X^\circ_{K/k,P}$ is a \emph{good} (respectively \emph{very good}) model of the equation \eqref{norm equation} if the first three (respectively all the four) of the following conditions are satisfied.
    \begin{enumerate}
        \item there is a proper $k$-morphism $\pi:X\to\P^1$ extending  $X^\circ_{K/k,P}\to \Af^1$ given by $(\mathbf{z},x)\mapsto x$;
        \item the fiber of $\pi$ over $\infty$ is non-degenerate, i.e. geometrically integral;
        \item the degenerate fibers lie over $\{P_1,\ldots,P_m\}$, and the fiber $X_{P_i}$ is birational over $k(P_i)$ to $X^\circ_{K/k,0}\otimes_kk(P_i)$
        \item the generic fibre $X_\eta$ of $\pi$ has geometric Picard group $\Pic(X_{\overline{\eta}})=\Z$ with trivial Galois action of $\Gamma_{k(x)}$;
    \end{enumerate}
\end{defn}

\begin{rem}\label{cyclic model}
    When $K$ is a cyclic extension of $k$, A. V\'arilly-Alvarado and B. Viray constructed a very good model of the equation \eqref{norm equation} in \cite{VAV-model}*{Theorem 1.1}. Their model was obtained by gluing partial compactifications of $X^\circ_{K/k,P}$ and $X^\circ_{K/k,\widetilde{P}}$ where $\widetilde{P}(x)=P(\frac{1}{x})\cdot x^{\deg{P}}$ is the reciprocal polynomial of $P$, cf. \cite{VAV-model}*{Lemma 5.1} for more details.
\end{rem}

\subsection{Vertical Brauer groups of good abelian normic bundles}
From now on, the degree $n$ Galois extension $K/k$ is supposed to be \emph{abelian}.
We are ready to compute the vertical unramified Brauer group of an abelian normic equation \eqref{norm equation} assuming the existence of a good model $X=X_{K/k,P}$.

In order to simplify the notation, we fix an embedding $k_i=k[x]/(P_i)\subset\overline{k}$ of the residue field of each closed point $P_i\in\PP^1$ defined by the polynomial equation $P(x)=0$. We write $K_i=Kk_i$ and $k'_i=K\cap k_i$. We denote by $G$ the abelian group $\Gal(K/k)$.  By Lemma \ref{galoistheory}, it has normal subgroups $G_i=\Gal(K_i/k_i)=\Gal(K/k'_i)$.
The homomorphism $\sigma\mapsto\sigma_{|K}$ identifies $\Gal(K(x)/k(x))$ with $G=\Gal(K/k)$ and their quotients $\Gal(k'_i(x)/k(x))$ with $G/G_i=\Gal(k'_i/k)$.

The epimorphisms $\varphi_i:\widehat{G}\rightarrow\widehat{G}_i$ between their Pontryagin duals are nothing but restriction homomorphisms between Galois cohomologies.
$$\xymatrix{
\widehat{G}=\HH^1(G,\Z/n\Z)\ar[r]^{\subset}\ar[d]^{\varphi_i}&\HH^1(k(x),\Z/n\Z)\ar[d]^{\mathrm{res}_{k_i(x)/k(x)}}\\
\widehat{G}_i=\HH^1(G_i,\Z/n\Z)\ar[r]^{\subset}&\HH^1(k_i(x),\Z/n\Z)
}$$

Let $Q$ be a closed point of $\PP^1$ corresponding to a monic irreducible polynomial $Q(x)$ (or $Q(x)=\frac{1}{x}$ if $Q=\infty$). The Witt residue $\partial_Q:\br(k(x))[n]\to\HH^1(k(Q),\Z/n\Z)$ at $Q$ is the one attached to the $Q$-adic valuation $\mathrm{ord}_Q:k(x)^*\to\Z$ for which the polynomial $Q(x)$ is a uniformiser.
We consider the cup product followed by  $\partial_Q$
\begin{align*}
    \HH^1(k(x),\Z/n\Z)\times\HH^1(k(x),\mu_n)&\buildrel{\smile}\over\To\HH^2(k(x),\mu_n)\buildrel{\simeq}\over\To\br(k(x))[n]\buildrel{\partial_Q}\over\To\HH^1(k(Q),\Z/n\Z)\\
    (\chi,(P)_n)&\xmapsto{\mbox{ }\mbox{ }\mbox{ }\mbox{ }\mbox{ }\mbox{ }\mbox{ }\mbox{ }\mbox{ }\mbox{ }\mbox{ }\mbox{ }\mbox{ }\mbox{ }\mbox{ }\mbox{ }\mbox{ }\mbox{ }\mbox{ }\mbox{ }\mbox{ }\mbox{ }\mbox{ }\mbox{ }\mbox{ }\mbox{ }\mbox{ }\mbox{ }\mbox{ }\mbox{ }\mbox{ }\mbox{ }\mbox{ }\mbox{ }\mbox{ }\mbox{ }\mbox{ }\mbox{ }\mbox{ }\mbox{ }\mbox{ }\mbox{ }\mbox{ }\mbox{ }\mbox{ }\mbox{ }\mbox{ }} \partial_Q(\chi\smile(P)_n)
\end{align*}
where $(P)_n$ is the class of a non-zero rational fraction  $P\in k(x)^*$ in $k(x)^*/k(x)^{* n}\simeq\HH^1(k(x),\mu_n)$.

When $\chi$ lies in the subgroup $\HH^1(G,\Z/n\Z)$ of $\HH^1(k(x),\Z/n\Z)$, it can be viewed as an element of $\HH^1(k,\Z/n\Z)$ under the identification $G=\Gal(K(x)/k(x))=\Gal(K/k)$. Then it follows from \cite{CTSkbook}*{Theorem 1.4.14, Equation (1.19)} and the compatibility of cup products with restriction homomorphisms between Galois cohomologies that
\begin{equation}\label{residueformula}
\partial_Q(\chi\smile(P)_n)=\mathrm{ord}_Q(P)\cdot\chi_Q,
\end{equation} where
$\chi_Q$ is the restriction of $\chi\in\HH^1(k,\Z/n\Z)$ in $\HH^1(k(Q),\Z/n\Z)$. We are going to apply this to the factors $P_i$ appearing in the equation \eqref{norm equation}. Only at the points $Q=P_i$ and $Q=\infty$ we may have non-trivial residues.

For each $1\leq i\leq m$,  we claim that the homomorphism
\begin{align*}
\widehat{G}&\to\br(k(x))\\
\chi_i&\mapsto\chi_i\smile (P_i)_n
\end{align*}
factorises through the quotient $\varphi_i:\widehat{G}\to\widehat{G}_i$ and thus defines a homomorphism
\begin{align*}
\iota_i:~\mbox{ }~\widehat{G}_i&\to\br(k(x))\\
\varphi_i(\chi_i)&\mapsto\chi_i\smile (P_i)_n.
\end{align*}
Indeed, suppose that we have elements $\chi_i$ and $\chi'_i$ of $\widehat{G}$ such that $\varphi_i(\chi_i)=\varphi_i(\chi'_i)$, then $\chi_i-\chi'_i\in \ker(\varphi_i)=\HH^1(\Gal(k'_i(x)/k(x)),\Z/n\Z)$. By Proposition \ref{polynorm}, the polynomial $P_i$ is a norm for $k'_i(x)/k(x)$. Then $\chi_i\smile (P_i)_n=\chi'_i\smile (P_i)_n$ since $(\chi_i-\chi'_i)\smile (P_i)_n=0$ according to \cite{CTSkbook}*{Proposition 1.3.8}.

\begin{lem}\label{brkp1}
      The homomorphism of abelian groups
      \begin{align*}
          \iota:~\mbox{ }~\mbox{ }\bigoplus_{i=1}^m \widehat{G}_i &\rightarrow \br(k(x))\\
          (\varphi_i(\chi_i))_{i=1}^m &\mapsto \sum_{i=1}^m \chi_i\smile (P_i)_n
      \end{align*}
      is a monomorphism whose image has zero intersection with the subgroup $\br(\PP^1)\subset\br(k(x))$.
\end{lem}

\begin{proof}
    According to \eqref{residueformula}, for each $1\leq j\leq m$ we have $\partial_{P_j}(\sum_{i=1}^m \chi_i\smile (P_i)_n)={\chi_j}_{P_j}=\varphi_j(\chi_j)$, which implies that $\ker(\iota)=0$.

    Suppose that $\sum_{i=1}^m \chi_i\smile (P_i)_n\in \br(\PP^1)$. It follows from the purity exact sequence
    \[
          0\To \br(\PP^1)\To \br(k(x))\buildrel{\partial_Q}\over\To \bigoplus_{Q}\mathrm{H}^1(k(P),\mathbb{Q}/\mathbb{Z})
      \]
    where $Q$ ranges over closed points of $\PP^1$, that $\varphi_j (\chi_j)=\partial_{P_j}(\sum_{i=1}^m \chi_i\smile (P_i)_n)=0$ for  $1\leq j\leq m$, which completes the proof.
\end{proof}

Recall that the degree of the irreducible factor $P_i$ of $P$ is  $d_i$. We define
\begin{align*}
\psi:\mbox{ }~\mbox{ }~\mbox{ }~\mbox{ }~\mbox{ }~\mbox{ }\widehat{G}^{m}&\to\widehat{G}\\
(\chi_1,\ldots,\chi_m)&\mapsto\sum_{i=1}^m d_i\chi_i
\end{align*}
and
\begin{align}\label{varphidef}
\begin{split}
\varphi=(\varphi_1,\ldots,\varphi_m):\widehat{G}\oplus\cdots\oplus\widehat{G}&\to\widehat{G}_1\oplus\cdots\oplus\widehat{G}_m
\\
(\chi_1,\ldots,\chi_m)&\mapsto(\varphi_1(\chi_1),\ldots,\varphi_m(\chi_m)).
\end{split}
\end{align}

Define $$B=B_{K/k,P}=\varphi(\ker(\psi))\subset\widehat{G}_1\oplus\widehat{G}_2\oplus\cdots\oplus\widehat{G}_m.$$

\begin{prop}\label{verticalbrauer}
      Let $X=X_{K/k,P}$  be a good model of the abelian normic equation \eqref{norm equation}. The morphism $\pi:X\to\PP^1$ induces a homomorphism $\pi_\eta^*:\br(k(x))\rightarrow \br(X_\eta)$. Then
      \begin{enumerate}
      \item as subgroups of $\br(X_\eta)$, we have
      $$\Br(X)\cap\pi_\eta^*\iota(\bigoplus_{i=1}^m\widehat{G}_i)=\pi_\eta^*\iota(B);$$
      \item the composition $$B\buildrel{\pi^*_\eta\iota}\over\To\br_{\textup{vert}}(X)\To \overline{\br}_{\textup{vert}}(X)$$ is an epimorphism of abelian groups.
      \end{enumerate}
\end{prop}

\begin{proof}
    We argue with the commutative diagram of exact sequences \cite{CTSkbook}*{Diagram (11.3)}
          \begin{equation}\label{Br_vert-diagram}
          \xymatrix{
              0 \ar[r] & \mathrm{Br}(X) \ar[r]                                    & \mathrm{Br}(X_\eta) \ar[rr]^(.35){\partial_Y}                      &  & {\bigoplus\limits_{Q}\bigoplus\limits_{Y\subset X_Q}\mathrm{H}^1(k(Q)(Y),\mathbb{Q}/\mathbb{Z})} \\
              0 \ar[r] & \mathrm{Br}(\PP^1) \ar[r] \ar[u]^{\pi^\ast} & \mathrm{Br}(k(x)) \ar[u]\ar[u]^{\pi_\eta^\ast} \ar[rr]^(.35){\partial_Q} &  & {\bigoplus\limits_{Q}\mathrm{H}^1(k(Q),\mathbb{Q}/\mathbb{Z})} \ar[u]
          }
          \end{equation}
    where $Q$ ranges over closed points of $\PP^1$ and $Y\subset X_Q$ ranges over the irreducible components of the fiber $X_Q$.
    The unnamed vertical homomorphism is $m_Y\mathrm{res}_{k(Q)(Y)/k(Q)}$, where $m_Y$ is the multiplicity of $Y$ in $X_Q$. According to the discussion at the beginning of Section \ref{normicvarietysubsection}, we have $m_Y=1$ for any $Y$ lying over any $Q$.

    The theory of Galois cohomology says that $$\ker(\mathrm{res}_{k(Q)(Y)/k(Q)})=\HH^1(\Gal(\overline{k(Q)}\cap k(Q)(Y)/k(Q)),\Q/\Z)$$
    once the algebraic closure $\overline{k(Q)}\cap k(Q)(Y)$  of $k(Q)$ in $k(Q)(Y)$ is an abelian extension of $k(Q)$. This algebraic closure is $k(Q)$ itself except for $Q=P_i$ in which case $k(P_i)=k_i$ and $\overline{k_i}\cap k_i(Y_{k_i,\tau})=Kk_i=K_i$ is an abelian extension of $k_i$ by Proposition \ref{algebraicclosureinfunctionfield} and Lemma \ref{galoistheory}. Hence $$\ker(\mathrm{res}_{k(P_i)(Y)/k(P_i)})=\HH^1(G_i,\Q/\Z)=\widehat{G}_i.$$

By the formula \eqref{residueformula}, we have

\begin{equation}\label{residuecomputation}
\partial_Q(\chi_i\smile (P_i)_n)=\left\{
\begin{array}{r@{\quad,\quad}l}
\varphi_i(\chi_i) & \mbox{if }Q=P_i,\\
-d_i\chi_i &\mbox{if }Q=\infty,\\
0 &\mbox{otherwise}.
\end{array} \right.
\end{equation}
     It follows from the diagram that an element $A\in\br (k(x))$ has image $\pi_\eta^\ast (A)$ lying in $\br(X)$ if and only if $\mathrm{res}_{k(Q)(V)/k(Q)}(\partial_Q(A))=0$ for all closed points $Q$ of $\PP^1$ and all irreducible components $Y$ of $X_Q$.
     This applied to
     $$A_0=\sum_{i=1}^m \chi_i\smile (P_i)_n$$
     shows that $\pi_\eta^\ast (A_0)\in\br(X)$ if and only if $$\psi(\chi_1,\ldots,\chi_m)=\sum_{i=1}^md_i\chi_i=0\in\HH^1(k(\infty),\Q/\Z)=\HH^1(k,\Q/\Z)$$ since $\ker(\mathrm{res}_{k(\infty)(X_\infty)/k(\infty)})$ is trivial and $\varphi_i(\chi_i)\in\widehat{G}_i=\ker(\mathrm{res}_{k(P_i)(V)/k(P_i)})$. This proves (1).

To prove (2), we recall by definition that
\begin{equation}\label{brvertdef}
\Br_\textup{vert}(X)=\br(X)\cap\pi^*_\eta(\br(k(x)))\subset\br(X_\eta).
\end{equation}
Then $\pi_\eta^*\iota(B)\subset\Br_\textup{vert}(X)$ by (1), in other words, the composition is well-defined. For the surjectivity of the composition, it remains to show that $\Br_\textup{vert}(X)$ is generated by $\pi_\eta^*\iota(B)$ and $\br_0(X)$.

Let $A\in\br(k(x))$ be such that $\pi^*_\eta(A)\in\br(X)$. Then $\partial_{P_i}(A)\in\widehat{G}_i$ for all $1\leq i\leq m$ and $\partial_{Q}(A)=0$ for all other closed points $Q$ of $\PP^1$. As $\varphi_i:\widehat{G}\to\widehat{G}_i$ is surjective, take $\chi_i\in\widehat{G}$ such that $\partial_{P_i}(A)=\varphi_i(\chi_i)=\partial_{P_i}(\chi_i\smile(P_i)_n)$. By setting
$$A_0=\sum_{i=1}^m \chi_i\smile (P_i)_n\in\iota(\bigoplus\widehat{G}_i)$$
as above, we obtain $\partial_Q(A-A_0)=0$ for all closed points of $Q$ of $\Af^1=\PP^1\setminus\{\infty\}$. Therefore $A-A_0\in\br(\Af^1)=\im(\br(k))$ by purity. As a consequence, we have
$$\pi^*_\eta(A)-\pi^*_\eta(A_0)\in\br_0(X)\subset \br(X)$$ and thus $\pi^*_\eta(A_0)\in\br(X)$.
Whence $A_0\in\iota(B)$ by (1) and we conclude that  $$\Br_\textup{vert}(X)=\pi_\eta^*\iota(B)+\br_0(X).$$
\end{proof}

\subsection{Brauer groups of very good abelian normic bundles}
When $X=X_{K/k,P}$ is furthermore a \emph{very good} model of the abelian normic equation \eqref{norm equation}, we can also compute the kernel of the composition in Proposition \ref{verticalbrauer}(2) and show that $\br_{\textup{vert}}(X)=\br(X)$. As a consequence, we obtain an explicit expression of $\overline{\br}(X)$.

\begin{thm}\label{maintheorem}
Let $X=X_{K/k,P}$  be a very good model of the abelian normic equation \eqref{norm equation}. Then
\begin{enumerate}
\item $\Br_\textup{vert}(X)=\Br(X)$;
\item the kernel of the composition
      \[
          B\buildrel{\pi^*_\eta\iota}\over\To\br_\textup{vert}(X)\To  \overline{\br}_\textup{vert}(X)
      \]
      is a cyclic subgroup $C$ generated by $c_\chi=(\varphi_i(\chi))_{i=1}^m$ for a certain  $\chi\in \widehat{G}$ such that $c_\chi$ has order the exponent $\exp(\bigoplus_{i=1}^m\widehat{G}_i)$, furthermore any $c_\chi$ of the same order is a generator;
\item the composition
$$B/((\varphi_i(\chi))_{i=1}^m)\To\br(X)\To\overline{\br}(X)$$
is an isomorphism of finite abelian groups;
\item the  homomorphism $B/((\varphi_i(\chi))_{i=1}^m)\To\br(X)$ is injective if  $c$ is further assumed to be a norm for $K/k$.
\end{enumerate}
\end{thm}

\begin{proof}
We consider the long exact sequence
      \begin{multline*}
          0 \longrightarrow \mathrm{Pic}(X_\eta) \longrightarrow \mathrm{Pic}(X_{\overline{\eta}})^{\Gamma_{k(x)}} \longrightarrow \mathrm{Br}(k(x)) \nonumber \\
           \longrightarrow \ker(\mathrm{Br}(X_\eta)\rightarrow \mathrm{Br}(X_{\overline{\eta}})) \longrightarrow \mathrm{H}^1(\Gamma_{k(x)},\mathrm{Pic}(X_{\overline{\eta}}))
      \end{multline*}
deduced from the Hochschild--Serre spectral sequence.

(1). As $X$ is a very good model of the abelian normic equation \eqref{norm equation}, the geometric Picard group $\mathrm{Pic}(X_{\overline{\eta}})$ of $X_\eta$ is isomorphic to $\Z$ as a $\Gamma_{k(x)}$-module.
According to the discussion at beginning of Section \ref{normicvarietysubsection}, the $k(x)$-variety $X_\eta$ is rational over $\bar{k}(\P^1)$, whence $\mathrm{Br}(X_{\overline{\eta}})=0$.
It follows from the sequence that $\pi_\eta^*:\br(k(x))\to\br(X_\eta)$ is an epimorphism. Therefore $\Br_\textup{vert}(X)=\Br(X)$ by \eqref{brvertdef}.

(2). It also follows from the sequence that the kernel of $\pi_\eta^*$ is a finite cyclic group. According to Lemma \ref{brkp1}, we know that $\iota(B)\cap \br(\PP^1)=\{0\}$. We are ready to apply Lemma \ref{cyclicgroup} to the homomorphism $\pi_\eta^*:\br(k(x))\to\br(X_\eta)$ as $\varphi:A\to A_3$ with subgroups $A_1=\iota(B)$ and $A_2=\br(\P^1)=\br(k)$ of $A$. We conclude that
$$\ker\left(\iota(B)\To\br(X_\eta)/\pi_\eta^*(\Br(k))\right)$$
is a cyclic group. This is nothing but
$$C=\ker\left(B\to\br_\textup{vert}(X)\to\overline{\br}_\textup{vert}(X)\right)$$
since $\br_\textup{vert}(X)\subset\br(X_\eta)$ and $\iota$ is injective by Lemma \ref{brkp1}.

Lemma \ref{group structure} applied to $A=\widehat{G}$ and $A_i=\widehat{G}_i$ guarantees the existence of a character $\chi\in\widehat{G}$ such that $c_\chi=(\varphi_i(\chi))_{i=1}^m$ has order $\exp(\bigoplus_{i=1}^m\widehat{G}_i)$.
It remains to show that $c_\chi$ lies in $C$ for any $\chi\in\widehat{G}$, then it automatically generates the cyclic subgroup $C\subset B\subset\bigoplus_{i=1}^m\widehat{G}_i$ because of its order. First of all, it is clear that $c_\chi\in B=\varphi(\ker(\psi))$ because $\psi(\chi,\ldots,\chi)=(\sum_{i=1}^md_i)\chi=\deg(P)\chi$ vanishes in $\widehat{G}$ as $n\mid\deg(P)$.
Moreover, as an element of $\br(k(X))$,  which contains $\br(X_\eta)\supset\im(\pi_\eta^*)$ as a subgroup, we have 
\begin{align*}
\pi_\eta^*\iota(c_\chi)&=\sum_{i=1}^m\chi\smile(P_i)_n\\
&=\chi\smile(\prod_{i=1}^m P_i)_n\\
&=\chi\smile(c^{-1}\N_{K/k}(\mathbf{z}))_n\\
&=\chi\smile(c^{-1})_n+\chi\smile(\N_{K/k}(\mathbf{z}))_n\\
&=\chi\smile(c^{-1})_n\in\im(\br(k))
\end{align*}
where the third equality follows from equations \eqref{norm equation} and \eqref{Pfactorisation} and where the last equality follows from \cite[Proposition 1.3.8]{CTSkbook}. Therefore the image of $c_\chi$ in $\overline{\br}_\textup{vert}(X)$ vanishes, in other words $c_\chi\in C$.

(3). The statement follows automatically from (1) and (2).

(4). Once $c$ is further assumed to be a norm for $K/k$, the image $\pi_\eta^*\iota(c_\chi)=\chi\smile(c^{-1})_n$  vanishes. Then the natural inclusion $\ker(\pi_\eta^*\iota)\subseteq C$ must be an equality and the injectivity follows.
\end{proof}

\begin{rem}
We see from the previous proof  that the kernel $C$ equals $\im(\varphi\circ\Delta)$ where
$\Delta:\widehat{G}\to\widehat{G}^m$ is the diagonal embedding. In other words, the cyclicity of the diagonal image of $\widehat{G}$ in $\bigoplus_{i=1}^m\widehat{G}_i$ is a necessary condition of the existence of a very good model $X_{K/k,P}$ of the equation \eqref{norm equation}.
\end{rem}

\section{Varieties with prescribed finite unramified Brauer groups}\label{section cyclic}

In this section, we suppose that  $K/k$ is a cyclic extension with $G=\Gal(K/k)\simeq\Z/n\Z$.
As mentioned in Remark \ref{cyclic model},  the existence of a very good model of the normic equation \eqref{norm equation} is guaranteed by \cite{VAV-model}. We apply results in the previous section  to obtain  unconditional statements.

\subsection{Unramified Brauer groups of cyclic normic equations}\label{Brnr-of-cyclicnormic}
We preserve the notation in the previous section. Now the subgroups $G_i$ of $G$ are all finite cyclic groups. In addition, we fix a generator $\sigma$ of $G$. Let $\chi_K:G\rightarrow \Z/n\Z$ be the character mapping $\sigma$ to $1\in\Z/n\Z$. Then $\chi_K$ generates $\widehat{G}$ and its quotient $\widehat{G}_i$. Any character $\chi_i\in\widehat{G}$ can be written as $\chi_i=\chi_K^{t_i}$ with $t_i\in\Z$. The exponent $t_i~\textup{mod}~ |G_i|$ corresponds uniquely to the element $\varphi_i(\chi_i)\in\widehat{G}_i$.
For a rational fraction $P\in k(x)$, we denote by $\left(\chi_K,P\right)\in\br(k(x))$ the class of the cyclic $k(x)$-algebra generated by $K(x)$ and a symbol $y$ subject to the relations $y^n=P$ and $\lambda y=y\sigma(\lambda)$ for any $\lambda\in K(x)$.
Then $\chi_K\smile (P)_n=\left(\chi_K,P\right)$ by \cite{CTSkbook}*{Formula (1.7) and Theorem 1.4.14}. In our case, this  $P\in k(x)$ will usually be the polynomial $P$ appearing in the equation \eqref{norm equation} or its monic irreducible factors $P_i$. In particular $\iota(\varphi_i(\chi_i))=\chi_i\smile (P_i)_n=\left(\chi_K^{t_i},P_i\right)=\left(\chi_K,P_i^{t_i}\right)$.

Let $f\in k[x]$ be a monic irreducible polynomial such that $K=k[x]/(f)$. As explained in Remark \ref{f-factorization}, it factorises over $k_i=k[x]/(P_i)$ as a product of irreducible polynomials of the same degree $n_i=[Kk_i:k_i]=|G_i|$. Then $n=[K:k]$ divides $[Kk_i:k]=[Kk_i:k_i][k_i:k]=n_i\deg(P_i)=n_id_i$.

We fix an isomorphism $\widehat{G}\simeq\Z/n\Z$ sending $\chi_K$ to $\overline{1}\in\Z/n\Z$, then its image $(\overline{1},\overline{1},\ldots,\overline{1})\in\bigoplus_{i=1}^m\widehat{G}_i\simeq\bigoplus_{i=1}^m\Z/n_i\Z$ has order equaling  the exponent of the group by Lemma \ref{cyclicexponent}.

We restate Theorem \ref{maintheorem}  as follows.

\begin{thm}\label{cyclic case}
Suppose that $K/k$ is a cyclic extension of degree $n$. Recall that $X^\circ_{K/k,P}\subset\Af^{n+1}$ is the affine hypersurface defined over $k$ by the equation
\begin{equation}\tag{\ref{norm equation}}
    \N_{K/k}(\mathbf{z})=P(x).
\end{equation}
Then
    \begin{align*}
        \frac{\left\{ (\overline{t}_i)_{i=1}^m\in \bigoplus_{i=1}^m (\Z/n_i\Z) \mid \sum_{i=1}^m d_i t_i=0  \mod n\right\}}{((\overline{1} , \overline{1}, \ldots, \overline{1})_{i=1}^m)}&\to\overline{\br}_\textup{nr}(X^\circ_{K/k,P}) , \\
        \quad (\overline{t}_i)_{i=1}^m &\mapsto \pi_\eta^*\left(\chi_K, \prod_{i=1}^r P_i^{t_i}\right)
    \end{align*}
is an isomorphism of finite abelian groups. \

    Moreover, when $c$ is a norm for $K/k$, then the isomorphism factorises as a monomorphism to $\br_\textup{nr}(X^\circ_{K/k,P})$ composed with the canonical projection.
\end{thm}

\begin{rem}
This theorem generalises the result \cite{VAV-chatletp-fold}*{Theorem 3.2} of V\'arilly-Alvarado and B. Viray on Ch\^atlet $p$-folds (i.e. the case where $n=p$ is a prime number). We would like to point out that, for the case where the splitting field of at least one irreducible factor $P_i$ is not linear disjoint from $K/k$, the statement of \cite{VAV-chatletp-fold}*{Theorem 3.2} is not accurate or is not able to cover all the exceptional cases. Related detail in the proof was not given. But their conclusion does not compatible with neither previous result by A. Skorobogatov nor our Theorem \ref{cyclic case}. For example, suppose that all factors $P_1,\ldots,P_m$ (with $m\geq4$) are distinct monic irreducible degree $p$ polynomials such that
\begin{itemize}
\item $K\simeq k[x]/(P_1)\simeq\cdots\simeq k[x]/(P_{m-2})$,
\item both $k[x]/(P_{m-1})$ and  $k[x]/(P_m)$ are  linear disjoint from $K$ over $k$.
\end{itemize}
The conclusion of \cite{VAV-chatletp-fold}*{Theorem 3.2} implies that the order of the Brauer group is $p^{m-2}$. However, in this case we have $n_1=\cdots=n_{m-2}=1$ and $n_{m-1}=n_m=p$, our Theorem \ref{cyclic case} implies that the Brauer group is $\Z/p\Z$ . When $n=p=2$, the Brauer group has already been computed by A. Skorobogatov in \cite{Skbook}. After the reduction in the first paragraph of \cite{Skbook}*{Section 7.1},  \cite{Skbook}*{Proposition 7.1.2} asserts that the Brauer group is generated by classes of quaternion algebras $\left(b,P_{m-1}\right)$ and $\left(b,P_m\right)$. But these represent the same class of order $2$, hence the Brauer group is $\Z/2\Z$.
\end{rem}

\begin{rem}\label{factorisation-determine-Br}
We see from Theorem \ref{maintheorem} that the unramified Brauer group of the equation \eqref{norm equation} is vertical once a very good model exists. From Theorem \ref{cyclic case}, we see that for cyclic extension $K/k$ the structure of the unramified Brauer is determined by the numbers $m$ and $(n_i)_{1\leq i\leq m}$ provided that $n=[K:k]$ divides the degree $d_i=\deg(P_i)$. The number $m$ is the number of factors of $P=\prod_iP_i$. Whereas the number $n_i$ is determined by the irreducible factorisation of $P_i$ over the extension $K$: the number of factors of $P_i=\prod_jP_{i,j}$ is $\frac{n}{n_i}$. Though different factors $P_i$ may have different residue fields $k_i=k[x]/(P_i)$ but they may have the same number $n_i=|\Gal(Kk_i/k_i)|$.
This factorisation information determines the structure of the Brauer group.
\end{rem}

\subsection{Cyclic normic equations of prescribed unramified Brauer groups}

It turns out that all finite abelian groups can be realised as the unramified Brauer group of a certain equation of the form \eqref{norm equation}.

\begin{thm}\label{cyclic construction}
    Let $B$ be a finite abelian group of exponent dividing $n$.
    Let $k$ be a number field containing a primitive $n$-th root of unity.

    Then there exist a cyclic extension $K/k$ of degree $n$ and a monic polynomial $P\in k[x]$ of degree divisible by $n$ such that there exists a monomorphism
    $$B\To\br_\textup{nr}(X^\circ_{K/k,P})$$ inducing an isomorphism
    $$B\simeq\overline{\br}_\textup{nr}(X^\circ_{K/k,P})$$ of finite abelian groups.
\end{thm}

\begin{proof}
We fix a primitive $n$-th root of unity $\zeta=\zeta_n\in k$. Choose an element $a\in k$ such that $f=x^n-a\in k[x]$ is irreducible. Define $K$ to be the cyclic extension $k[x]/(f)$. We fix a root $\sqrt[n]{a}$ of $f$ in $K$.

Write  $B$ as a direct sum $\bigoplus_{i=1}^m \Z/n_i\Z$ with $n_i\mid n$. In addition, we set $n_0=n$ and consider $\Z/n_0\Z\oplus B$.  For $0\leq i\leq m$, we aim to find distinct  polynomials $P_i\in k[x]$ such that
\begin{itemize}
\item $P_i$ is monic irreducible,
\item $d_i=\deg(P_i)=n$,
\item $[Kk_i:k_i]=n_i$ where $k_i=k[x]/(P_i)$.
\end{itemize}
Once this is achieved,  for the affine hypersurface $X^\circ_{K/k,P}\subset\Af^{n+1}$ defined by the equation \eqref{norm equation} with $P=\prod_{i=0}^mP_i$, Theorem \ref{cyclic case} asserts that (note that the constraint $\sum_{i=0}^md_it_i=0\mod n$ is superfluous since $d_i=n$)
$$\frac{\Z/n_0\Z\oplus B}{((\overline{1} , \overline{1}, \ldots, \overline{1})_{i=0}^m)}{\simeq} \overline{\br}_\textup{nr}(X^\circ_{K/k,P}).$$
It is clear that the diagonal homomorphism
\begin{align*}
\delta:~\mbox{ }~\Z/n\Z&\to\Z/n_0\Z\oplus B\\
t~\textup{mod}~ n&\mapsto (t~\textup{mod}~ n_i)_{i=0}^m
\end{align*}
has a retraction $\textup{pr}_0:(t_i\mod n_i)_{i=0}^m\mapsto t_0\mod n$. Therefore
$$\frac{\Z/n_0\Z\oplus B}{((\overline{1} , \overline{1}, \ldots, \overline{1})_{i=0}^m)}=\frac{\Z/n_0\Z\oplus B}{\im(\delta)}\simeq\ker(\textup{pr}_0)=B,$$
which allows us to conclude.

It remains to construct a suitable polynomial $P_i$ for each given positive integer $n_i\mid n$ with $0\leq i\leq m$, and make sure that they are distinct to each other.

Take $Q_i=x^{n_i}-u_i\in k[x]$ such that $Q_i-(\sqrt[n]{a})^{n_i}$ is irreducible over $K$. The existence of such a $u_i\in k$ is guaranteed by the forthcoming Lemma \ref{construct irreducible}.  Set $P_i=Q_i^\frac{n}{n_i}-a$, then $\deg(P_i)=n$ and it is monic. Moreover,  we can choose distinct $u_i\in k$ to make sure that these $P_i$ are distinct. Each $P_i$ is irreducible over $k$ by Lemma \ref{irreducibility} applied with $L=K$, $\theta=(\sqrt[n]{a})^{n_i}$, $h=Q_i$ and $g=x^\frac{n}{n_i}-a$, where $g$ is irreducible over $k$ since $f=x^n-a$ is irreducible.
Over $K$, we have the following factorization
$$P_i=Q_i^\frac{n}{n_i}-a=(Q_i-\zeta^{n_i}(\sqrt[n]{a})^{n_i})(Q_i-\zeta^{2n_i}(\sqrt[n]{a})^{n_i})\cdots(Q_i-\zeta^{(\frac{n}{n_i}-1)n_i}(\sqrt[n]{a})^{n_i})(Q_i-(\sqrt[n]{a})^{n_i})$$
As a $\Gal(K/k)$-conjugate of $Q_i-(\sqrt[n]{a})^{n_i}$, each term $P_{i,j}=Q_i-\zeta^{jn_i}(\sqrt[n]{a})^{n_i}$ for $1\leq j\leq\frac{n}{n_i}$ on the right hand side is irreducible over $K$.

If we denote $k_i=k[x]/(P_i)$, Remark \ref{f-factorization} says that $f=x^n-a$ factorises over $k_i$
as $f=\prod_{j=1}^{\frac{n}{n'_i}}f_j$ with irreducible factors $f_j$ all of degree $n'_i=[Kk_i:k_i]$.
To complete the proof, it suffices to show that $n_i=n'_i$. We compute $K\otimes_kk_i$ in two ways as follows.
\begin{align*}
K\otimes_kk_i&=K\otimes_k\left(k[x]/(P_i)\right)\simeq K[x]/(P_i)\simeq\prod_{j=1}^{\frac{n}{n_i}}K[x]/(P_{i,j})\\
K\otimes_kk_i&=\left(k[x]/(f)\right)\otimes_kk_i~\simeq~ k_i[x]/(f)~\simeq\prod_{j=1}^{\frac{n}{n'_i}}k_i[x]/(f_{j})
\end{align*}
As both $k$-algebras are finite products of field extensions of $k$, the number of factors must equal, i.e. $n_i=n'_i$.
\end{proof}

We close this section by the following lemma, the simple version of which has already applied  for the existence of certain irreducible polynomials in the previous proof.

\begin{lem}\label{construct irreducible}
Let $k$ be a number field containing a primitive $n$-th root of unity and $K=k(\sqrt[n]{a})$ be the cyclic extension defined by irreducible polynomial $f=x^n-a\in k[x]$.

Then, for any positive integer $d\mid n$, there exist infinitely many elements $u\in k$ such that $x^d-u-(\sqrt[n]{a})^d\in K[x]$ is irreducible.

    Moreover, for a given positive integer $m$, let $\Omega_k^{a,m,n}\subset\Omega_k$ be the finite subset of non-archimedean places $v$ such that
    \begin{itemize}
    \item[\textup{(v1)}] $\mathrm{ord}_v(a)=1$,
    \item[\textup{(v2)}] $\mathrm{char}(\F_v)>nm+2$.
    \end{itemize}
    Then we can require $u$ to satisfy in addition that for all $v\in\Omega_k^{a,m,n}$
    \begin{itemize}
    \item[\textup{(u1)}] $\mathrm{ord}_v(u)=0$,
    \item[\textup{(u2)}] the reduction $\mathrm{mod}~v$ of the equation $x^d-u=0$ has roots in $\overline{\F}_v^*$ outside any given finite subset of at most $nm$ elements.
    \end{itemize}

\end{lem}

\begin{rem}
The second part (the advanced version) of the lemma is applied only in the next section. Though the subset $\Omega_k^{a,m,n}$ may be empty in general, it is not the case in our application.
\end{rem}

\begin{proof}
By Chebotarev's density theorem, we choose a non-archimedean place $w$ of $k$ which splits completely in $K$.
Let $w'$ be a place of $K$ lying over $w$. As $k$ is dense in $k_w=K_{w'}$, there exist an element $u\in k$ such that $\mathrm{ord}_{w}(u+(\sqrt[n]{a})^d)=1$. The polynomial $x^d-u-(\sqrt[n]{a})^d$ is irreducible over $K$ by Eisenstein's criterion.
Removing any finite subset from $k$ does not change the relevant density required in the argument, so infinitely many such $u$ exist.

Moreover, when we choose the place $w$ we can require in addition that $\mathrm{ord}_w(a)=0$, then $w\notin\Omega_k^{a,m,n}$. For any finite subset $k_0$ of $k$, the image of the diagonal embedding
$$k\setminus k_0\To\prod_{v\in\Omega_k^{a,m,n}\cup\{w\}}k_v$$
is dense by weak approximation. Denote by $E_v\subset\overline{\F}_v$ an exceptional subset of at most $nm$ elements. Then $E_v^d=\{e^d\mid e\in E_v\}$ has at most $nm$ elements and therefore $\F_v^*\setminus E_v$ is non-empty. Once the reduction of $u$ $\mathrm{mod}~v$ lies in $\F_v^*\setminus E_v$, the condition \textup{(u2)} is satisfied. Then \textup{(u2)} together with \textup{(u1)} defines an open subset of $k_v$ for each $v$, we obtain infinitely many elements $u\in k$ satisfying both conditions by weak approximation.
\end{proof}

\section{Brauer classes precisely obstructing the Hasse principle}\label{section maintheorem2}

In this last section, we discuss the existence of the Brauer--Manin obstruction to the Hasse principle on very good models of the cyclic normic equation \eqref{norm equation}. The main result of this paper is the following theorem.

\begin{thm}\label{thm2}
Let $B$ be a finite abelian group of exponent dividing $n$ and   $0\neq B_0\subset B$ be a non-zero subgroup.
Let $k$ be a number field containing a primitive $n$-th root of unity.
Then there exists a rationally connected variety $X$ defined over $k$ satisfying the following conditions
\begin{enumerate}
  \item there is a monomorphism $B\To\br(X)$ inducing an isomorphism $B\simeq\overline{\br}(X)$ of finite abelian groups,
  \item the subgroup $B_0$ precisely obstructs the Hasse principle, i.e. $X(\mathbf{A}_k)^{B'}=\varnothing $ for a subgroup $B'\subset B$ if and only if $B_0\subset B'$.
\end{enumerate}
\end{thm}
In particular, when $B_0=B$ we obtain an immediate corollary, which generalises \cite{Berg-etal}*{Theorem 1.1}.

\begin{cor}
Let $B$ be a non-zero finite abelian group of exponent dividing $n$.
Let $k$ be a number field containing a primitive $n$-th root of unity.

Then there exists a rationally connected variety $X$ defined over $k$ such that its Brauer group $\br(X)$ contains $B$ as a subgroup which maps isomorphically onto $\overline{\br}(X)$. Moreover,  the whole Brauer group is required to obstruct the Hasse principle, i.e. $X(\mathbf{A}_k)^B=\varnothing $ while $X(\mathbf{A}_k)^{B'}\neq \varnothing $ for any proper subgroup $B'\subsetneq B$.
\end{cor}

We are going to extend the argument of \cite{Berg-etal}*{Theorem 1.1} to prove Theorem \ref{thm2}.

\subsection{Reinterpretation of the Brauer--Manin obstruction}

Let $X$ be a smooth, projective and geometrically integral variety over a number field $k$.
For a finite subgroup $B\subset \br(X)$ we denote by $\widehat{B}=\mathrm{Hom}(B,\Q/\Z)$ its Pontryagin dual.
For each place $v\in \Omega_k$, we define a map
\begin{align*}
  \lambda_v:X(k_v)&\rightarrow\widehat{B}\\
  P_v&\mapsto \mbox{``}b\mapsto \inv_v (b(P_v))\mbox{''}
\end{align*}
where the evaluation $b(P_v)$ of a Brauer element $b$ is its image by the induced homomorphism $P_v^*:\br(X)\to\br(k_v)$. Denote by $\Lambda_v=\Lambda_v(X,B)$ the image $\im(\lambda_v)$. By a standard good reduction argument, the subset $\Lambda_v$ is $\{0\}$ for all but finitely many places $v$. Thus we are able to define a map $\lambda=\sum_{v\in\Omega_k}\lambda_v:X(\mathbf{A}_k)\rightarrow \widehat{B}$ and a subset $\Lambda=\Lambda(X,B)=\im(\lambda)\subset\widehat{B}$.
We have the following criteria for the existence of  Brauer--Manin obstruction.

\begin{lem}\label{criteriaBM}
Let $B_0$ and $B'$ be subgroups of $B$.
\begin{enumerate}
      \item The Brauer--Manin subset $X(\Ad_k)^B=\varnothing $ if and only if $0\notin \Lambda$.
      \item The Brauer--Manin subset $X(\Ad_k)^{B'}=\varnothing $ if and only if $\Lambda\cap \ker(\widehat{B}\rightarrow \widehat{B}')=\varnothing$.
      \item\label{criteriaBM3} If $\Lambda=\widehat{B}\setminus \ker(\widehat{B}\rightarrow\widehat{B}_0)$, then $X(\Ad_k)^{B'}=\varnothing $ if and only if $B_0\subset B'$.
\end{enumerate}
\end{lem}

\begin{proof}
(1) is copied from \cite[Lemma 3.1(a)]{Berg-etal}.

(2) is \cite[Lemma 3.1(b)]{Berg-etal} except for the case where $B'=B$ which has been dealt with in (1).

We claim that $\ker(\widehat{B}\rightarrow \widehat{B}')\subset \ker(\widehat{B}\rightarrow\widehat{B}_0)$ if and only if $B_0\subset B'$.

Indeed, if there exists an element $b\in B_0\setminus B'$,  we can choose a homomorphism $\bar{f}:B/B'\to\Q/\Z$ such that $\bar{f}(b~\mathrm{mod}~B')\neq0$. Let $f$  be the composition of $\bar{f}$ and the canonical projection $B\to B/B'$, then $f\in\ker(\widehat{B}\rightarrow \widehat{B}')\setminus \ker(\widehat{B}\rightarrow\widehat{B}_0)$ by definition. Conversely, if there exists an element $f\in\ker(\widehat{B}\rightarrow \widehat{B}')\setminus \ker(\widehat{B}\rightarrow\widehat{B}_0)$, then  $f:B\to\Q/\Z$ is such that $f_{|B'}=0$ but $f_{|B_0}\neq0$. Take $b\in B_0$ with $f(b)\neq0$, then $b\in B_0\setminus B'$.

(3) follows immediately from the claim.
\end{proof}

\subsection{A refinement}

The key observation Lemma \ref{criteriaBM}(\ref{criteriaBM3}) together with Theorem \ref{cyclic construction} leads  to a proof of Theorem \ref{thm2}: during the construction of the cyclic normic hypersurface $X^\circ_{K/k,P}$ in the proof of Theorem \ref{cyclic construction}, it suffices to make sure in addition that the corresponding subset $\Lambda$ of its very good model $X$ satisfies $\Lambda=\widehat{B}\setminus \ker(\widehat{B}\rightarrow\widehat{B}_0)$ . But it turns out that we can prove a bit more: we are able to control the evaluation of the Brauer elements at finitely many places. We generalise \cite[Theorem 7.1]{Berg-etal} to cyclic normic hypersurfaces defined by the equation \eqref{norm equation}.

\begin{thm}\label{mainprop}
Let $B$ be a finite abelian group of exponent dividing $n$.
Let $k$ be a number field containing a primitive $n$-th root of unity.
Let $\Lambda_1,\ldots,\Lambda_r$ be non-empty subsets of $\widehat{B}$.

Then there exist non-archimedean places $w_1,\ldots, w_r$ of $k$ not dividing $n$
and a rationally connected $k$-variety $X$ which is a smooth projective model of the cyclic normic equation \eqref{norm equation} such that
  \begin{enumerate}
  \item there is a monomorphism $B\To\br(X)$ inducing an isomorphism $B\simeq\overline{\br}(X)$,
  \item for $t=1,\ldots,r$, the image $\im(\lambda_{w_r}:X(k_{w_r})\rightarrow \widehat{B})=\Lambda_r$,
  \item for all places $v\notin\{w_1,\ldots,w_r\}$, the image $\im(\lambda_{v}:X(k_{v})\rightarrow \widehat{B})=\{0\}$.
  \end{enumerate}
In particular, the variety $X$ has rational points locally everywhere.
\end{thm}
\begin{proof}[Proof of Theorem \ref{thm2} assuming Theorem \ref{mainprop}]
We have $\ker(\widehat{B}\rightarrow \widehat{B}_0)=\widehat{B/B_0}$ by left exactness of $\mathrm{Hom}(-,\Q/Z)$.
Then the condition that $B_0\neq0$ implies $\ker(\widehat{B}\rightarrow \widehat{B}_0)\subsetneq\widehat{B}$ by counting.
Applying Theorem \ref{mainprop}  with $r=1$ and $\Lambda_1=\widehat{B}\setminus \ker(\widehat{B}\rightarrow \widehat{B}_0)\neq\varnothing$, we obtain $\Lambda=\sum_{v\in\Omega_k}\im(\lambda_v)=\Lambda_1$. The proof is completed by Lemma \ref{criteriaBM}(\ref{criteriaBM3}).
\end{proof}

The rest aims to prove Theorem \ref{mainprop}.

\subsection{Preparation for the proof}
\begin{lem}[variant of \cite{Berg-etal}*{Proposition 4.1}]\label{finitefield}
  Let $m$ be positive integers.
  Let $\mathbb{F}$ be a finite field.
  For $0\leq i\leq m$, let $q_i\in\mathbb{F}[x]$ be a polynomial of degree $n_i$ such that $q=\prod_{i=0}^{m}q_i$ is separable. Assume that $\mathrm{char}(\mathbb{F})$ does not divide $N=\prod_{i=0}^mn_i$ and that $|\F|$ is sufficiently large so that  
$$|\mathbb{F}|+1-2g\sqrt{|\mathbb{F}|}>(1+\sum_{i=0}^{m}n_i)N$$
where $g=1+\frac{1}{2}(\sum_{i=0}^{m} n_i-m-3)N$.

Then for an arbitrary family of constants $(\varepsilon_i)_{i=0}^m\in\prod_{i=0}^m\mathbb{F}^*$, there exists an element $\bar{u}\in\F$ such that $q_i(\bar{u})\varepsilon_i\in \mathbb{F}^{* n_i}$ for all $i$.
\end{lem}

\begin{proof}
 Consider the projective variety $C\subset\PP^{m+2}$ defined over $\F$ by a system of homogeneous equations $y_i^{n_i}=\varepsilon_i q_i(x/z)z^{n_i}$ for $i=0,\ldots,m$. Under our assumption on the separability of $q$ and the divisibility of $N$,  the variety $C$ is a smooth complete intersection of dimension $1$ by jacobian criterion. It must be a geometrically integral curve. By computing the degree of the canonical bundle, we see that $C$ is of genus $g=1+\frac{1}{2}(\sum_{i=0}^{m} n_i-m-3)N$.
  
  The map $(x:y_0:\cdots:y_m:z)\mapsto (x:z)$ defines a dominant rational map from the curve $C$ to $\PP^1$ which extends to a morphism $\pi$ between projective algebraic curves. It is a ramified covering of $\PP^1$ of degree $N$. Consider the closed subscheme $Z=\mathrm{Spec}(\F[x]/(q))\cup\{\infty\}$ of $\PP^1$. The conditions $q_i(\bar{u})\varepsilon_i\in \mathbb{F}^{* n_i}$ are simultaneously satisfied for all $i$ once $(\bar{u}:y_0:\cdots:y_m:z)$ is the coordinates of a certain $\F$-point of $C\setminus\pi^{-1}(Z)$. If $|C(\mathbb{F})|>(1+\sum_{i=0}^{m}n_i)N\geq |\pi^{-1}(Z)(\F)|$, then such $\bar{u}\in\F$ exists. We conclude by the Hasse--Weil bound.
\end{proof}

\begin{prop}[variant of \cite{Berg-etal}*{Proposition 7.3}]\label{surj}
  Let $B$ be a finite abelian group of exponent dividing $n$ and $k$ be a number field containing a primitive $n$-th root of unity.
  Then there exist non-archimedean places $w_1,\ldots,w_r$ of $k$ not dividing $n$ and a rationally connected $k$-variety $X'=X'_{k(\sqrt[n]{a})/k,P}$ which is a very good model of the cyclic normic equation \eqref{norm equation} such that 
  \begin{enumerate}
      \item there is a monomorphism  $B\To\br(X')$ inducing an isomorphism $B\simeq \overline{\br}(X')$;
      \item all archimedean completions of  $k$ contain $\sqrt[n]{a}$;
      \item the induced map $\lambda'_{w_t}:X'(k_{w_t})\rightarrow\widehat{B}$ is surjective for all $1\leq t\leq r$;
      \item the image $\im(\lambda'_v:X'(k_{v})\rightarrow\widehat{B})$ contains $0$ for all places $v$ of $k$.
  \end{enumerate}
\end{prop}

\begin{proof}
  Write $B$ as a direct sum $\bigoplus_{i=1}^{m}\Z/n_i\Z$ with $n_i\mid n$. In addition, we set $n_0=n$.
  
  Let $w_1,\ldots,w_r$ be non-archimedean places of $k$ not dividing $n$ such that $|\mathbb{F}_{w_t}|>nm+2$ is sufficiently large such that the inequality of Lemma \ref{finitefield} holds. 

  By weak approximation, we choose an element $a\in k^*$ such that 
  \begin{itemize}
  \item all archimedean completions of $k$ contain a fixed (then all) $n$-th root(s) $\sqrt[n]{a}$ i.e. the conclusion (2), for instance it suffices to require $a$ to be totally positive when $n$ is even;  
  \item the $w_t$-adic valuation $\textup{ord}_{w_t}(a)=1$ for all $1\leq t\leq r$.
  \end{itemize}
  Then $x^n-a\in k[x]$ is irreducible by Eisenstein's criterion.
  We fix a primitive $n$-th root of unity $\zeta=\zeta_n\in k$ and set $K=k(\sqrt[n]{a})$. The character $\chi_K$ of $G=\Gal(K/k)$ defined at the beginning of Section \ref{Brnr-of-cyclicnormic} coincides with the character $\chi_a$ defined in Section \ref{local-inv} with $\sigma:\sqrt[n]{a}\mapsto\zeta\sqrt[n]{a}$. The Brauer classes defined respectively are related as follows: the one $\left(\chi_a,b\right)$  in Section \ref{local-inv} is the evaluation at a certain local rational point $\Theta_v$ of $X'$ of the one $\pi_\eta^*\left(\chi_K,Q\right)$ in Section \ref{Brnr-of-cyclicnormic} for a certain rational function $Q\in k(x)$, they are identified via the substitution $b=Q(x(\Theta_v))$ of the $x$-coordinate of $\Theta_v$ if it is not a zero of $Q$.

  Now we run the proof of Theorem \ref{cyclic construction} to carefully choose a  monic polynomial $P'$, and we take a very good model $X'$  given by \cite[Theorem 1.1]{VAV-model} of the cyclic normic equation  \eqref{norm equation}. The conclusion (1) follows as argued in Theorem \ref{cyclic construction}.
  
  Recall that $P'$ has irreducible factorisation $P'=\prod_{i=0}^mP'_i$ with $P'_i={Q'}_i^{\frac{n}{n_i}}-a$ where $Q'_i=x^{n_i}-u_i\in k[x]$.
  
  It remains to prove the conclusions (3) and (4), where the evaluation of Brauer classes in $B$ has to be taken into account. According to Theorems \ref{cyclic case} and \ref{cyclic construction}, the image of $B$ in $\br(X')$ is generated by classes of cyclic algebras $\pi_\eta^*\left(\chi_a,P'_i\right)$ with $1\leq i\leq m$. 
  
  The reciprocal polynomial $\widetilde{P}'$ of $P'$ has constant term $1$ which is trivially a norm for $K/k$. According to Remark \ref{cyclic model}, the model $X'$ has a $k$-rational point $\Theta$ lying over $\infty\in\mathbb{P}^1$. Note that $\deg(P'_i)=n$, the reciprocal polynomial $\widetilde{P}'_i(x)=P'_i(\frac{1}{x})x^n$ has constant term $1$ since $P'_i$ is monic. We see that the Brauer class $\left(\chi_a,P'_i\right)=\left(\chi_a,\widetilde{P}'_i(\frac{1}{x})x^n\right)=\left(\chi_a,\widetilde{P}'_i(\frac{1}{x})\right)$ since the character $\chi_a$ is of order dividing $n$. Because the $x$-coordinate of $\Theta$ is $\infty$, the evaluation of $\pi_\eta^*\left(\chi_a,P'_i\right)$ at $\Theta\in X'(k_v)$ equals $\left(\chi_a,\widetilde{P}'_i(\frac{1}{\infty})\right)=\left(\chi_a,1\right)=0\in\br(k_v)$. This indicates that the image of $\Theta$ by the map $\lambda'_v:X'(k_v)\to\widehat{B}$ is $0$, from which the conclusion (4) follows.  
  
  To reach the conclusion (3), while running the proof of Theorem \ref{cyclic construction}, we need to require $Q'_i$ to satisfy extra conditions
  \begin{itemize}
  \item for $0\leq i\leq m$ and $1\leq t\leq r$, the element $u_i\in \mathcal{O}_{k_{w_t}}$;
  \item the reduction $\mathrm{mod}~w_t$ of the polynomial $Q'=\prod_{i=0}^m Q'_i$ is separable over $\F_{w_t}$.
  \end{itemize}
  These are  done by replacing the application of Lemma \ref{construct irreducible} by its advanced version, where we should notice that $\Omega^{a,m,n}_k$ contains $\{w_t\mid 1\leq t\leq r\}$ by our choice of $a$ and the lower bound of $|\F_{w_t}|$.

We are about to prove the surjectivity of the maps $\lambda'_{w_t}:X'(k_{w_t})\to\widehat{B}$ for $1\leq t\leq r$. We write $v=w_t$ to simplify the notation. Belonging to the image in $\br(X')$ of the direct factor $\Z/n_i\Z\subset B$, the Brauer class $\left(\chi_a,P'_i\right)$ is of order dividing $n_i$. For any $\beta\in\widehat{B}$, we denote $\tau_{i,\beta}=\beta(\left(\chi_a,P'_i\right))\in \frac{1}{n}\Z/\Z\subset\Q/\Z$, which is also of order dividing $n_i$. We look for $\Theta_v\in X'(k_v)$ such that $\inv_v((\pi_\eta^*\left(\chi_a,P'_i\right))(\Theta_v))=\tau_{i,\beta}$, which implies $\lambda'_v(\Theta_v)=\beta$ since $\{\pi_\eta^*\left(\chi_a,P'_i\right)\mid 1\leq i\leq m\}$ generates $B$.

Let $\gamma$ be a generator of the cyclic group $\F_v^*$. Recall in Section \ref{local-inv} that $n$ divides $|\F_v|-1$ and $\log$ is an isomorphism, it follows that the element $\log(\gamma^{\frac{|\F_v|-1}{n}})$ generates $\frac{1}{n}\Z/\Z$. Let $t$ be an integer such that $\log(\gamma^{t\frac{|\F_v|-1}{n}})=\tau_{i,\beta}$. By comparing the order, we see that $n$ divides $tn_i$. For $1\leq i\leq m$, set $\varepsilon_i=\gamma^{\frac{tn_i}{n}}\in\F_v^*$, then $$\log((\varepsilon_i^{\frac{n}{n_i}})^{\frac{|\F_v|-1}{n}})=\tau_{i,\beta}.$$
Let $\varepsilon_0\in \mathbb{F}_v^*$ be such that $$\varepsilon_0\cdot\prod_{i=1}^{m}\varepsilon_i^{\frac{n}{n_i}}=1.$$

As $\mathrm{ord}_v(a)=1$ by the choice of $a$, the reduction $\overline{P'_i}=\overline{Q'_i}^{\frac{n}{n_i}}\in\mathbb{F}_v[x]$.
Lemma \ref{finitefield} applied to $\F=\F_v$ and $q_i=\overline{Q'_i}$ implies the existence of an element $u\in \mathcal{O}_{k_v}^\times$ such that its reduction $\bar{u}\in\F_v^*$ satisfies $\overline{Q'_i}(\bar{u})\varepsilon_i\in \mathbb{F}_v^{* n_i}$ for $0\leq i\leq m$. Then 
\begin{equation}\label{n-power}
\overline{P'_i}(\bar{u})\varepsilon_i^{\frac{n}{n_i}}\in \mathbb{F}_v^{* n}
\end{equation} 
and $P'_i(u)$ is a $v$-adic unit. By taking product for $0\leq i\leq m$, we get $\overline{P'}(\bar{u})\in \mathbb{F}_v^{\times n}$, then $P'(u)\in \mathcal{O}_{k_v}^{\times n}$ by Hensel's lemma. As $P'(u)$ is a norm, there exists a local point $\Theta_v\in X'(k_v)$ having $u$ as its $x$-coordinate.
Then 
\begin{align*}
\inv_v((\pi_\eta^*\left(\chi_a,P'_i\right))(\Theta_v))&=\inv_v(\left(\chi_a,P'_i(u)\right))\\
&=\log(\left(\frac{1}{\overline{P'_i}(\bar{u})}\right)^{\frac{|\F_v|-1}{n}})\\
&=\log((\varepsilon_i^{\frac{n}{n_i}})^{\frac{|\F_v|-1}{n}})\\
&=\tau_{i,\beta}
\end{align*}
where the second equality follows from Proposition \ref{inv} and the third follows from \eqref{n-power}. This completes the proof of (4).
\end{proof}

\subsection{Proof of the refinement}

Finally, we are able to adapt the idea of the proof of \cite{Berg-etal}*{Theorem 7.1} to our context.

\begin{proof}[Proof of Theorem \ref{mainprop}]

  Proposition \ref{surj} gives us an element $a\in k^*$, non-archimedean places $w_1,\dots,w_r$, and a very good model $\pi':X'\to\PP^1$ satisfying the conclusions (1)(2)(3)(4) of Proposition \ref{surj}. In particular, the set of local rational points $X'(k_v)$ is non-empty for all places $v$ and therefore $\pi'(X'(k_v))$ contains a non-empty $v$-adic open set of $\PP^1(k_v)$.
  
  The strategy is to find a suitable morphism $h:\PP^1\to\PP^1$ and define $X$ to be the pull-back of $\pi'$ along $h$. Then $h$ will help us to control the images $\im(\lambda_v:X(k_v)\to\widehat{B})$ which is naturally contained in $\im(\lambda'_v:X'(k_v)\to\widehat{B})$.
  
  \textbf{Reduction on non-archimedean places.}
  \begin{itemize}
  \item[]
  By a standard good reduction argument, there are only finitely many places $v$ such that $\lambda'_v:X'(k_v)\to\widehat{B}$ is possibly non-zero.  Without lost of generality, we may assume that they are $w_1,\dots,w_r,w_{r+1},\ldots,w_s$, which are non-archimedean according to Proposition \ref{surj}(2). For $r<t\leq s$, we set $\Lambda_t=\{0\}$. Then for $1\leq t\leq s$, we have $\Lambda_t\subset\im(\lambda'_{w_t})$. Moreover, we have $\im(\lambda'_{v})=\{0\}$ for any $v\in\Omega_k\setminus\{w_1,\ldots,w_s\}$. 
  
  In order to prove Theorem \ref{mainprop}, it suffices to prove it for $s$ in place of $r$. 
  \end{itemize}

  Denote by $Z'=\Spec(k[x]/(P'))\subset\PP^1$ the degeneracy locus of $\pi':X'\to\PP^1$.   For $1\leq t\leq s$, define 
  \begin{align*}
    \mathcal{U}_{w_t}& =\pi'(\lambda'^{-1}_{w_t}(\Lambda_t)\cap (X'\setminus \pi^{-1}(Z'))(k_{w_t})) \\
     &=\bigcup_{\lambda\in\Lambda_t}\mathcal{U}_{w_t,\lambda}
  \end{align*}
  where $\mathcal{U}_{w_t,\lambda}=\pi'(\lambda'^{-1}_{w_t}(\lambda)\cap (X'\setminus \pi^{-1}(Z'))(k_{w_t}))\subset \PP^1(k_{w_t})$ is a $w_t$-adically open and non-empty subset since the evaluation of a Brauer class is locally constant. A $k$-rational point $c\in\PP^1\setminus Z'$ lies in $\mathcal{U}_{w_t}$ signifies that the fiber $X'_c$ has at least one $k_{w_t}$-point $\Theta_{w_t}$ and $\lambda'_{w_t}(\Theta_{w_t})\in\Lambda_t$.

  \textbf{Reduction on the fiber at infinity.} 

  \begin{itemize}
  \item[]
  Without lost of generality, we may assume that the point $\infty\in\mathcal{U}_{w_t}$ for $1\leq t \leq s$, otherwise we may proceed as follows.

  By weak approximation, we choose $c\in\PP^1(k)\setminus\{\infty\}$ such that $c\in\mathcal{U}_{w_t}$ for $1\leq t \leq s$. The rational map $\phi:\PP^1\to\PP^1$ defined by $x\mapsto \frac{cx}{x+c}$ is an automorphism of $\PP^1$ mapping $\infty$ to $c$. We will replace $X'$ by its pull-back $X''=X'\times_{\P^1,\phi}\PP^1$ which is a smooth projective  model of the equation
  $$\N_{k(\sqrt[n]{a})/k}(\mathbf{z})=P'(\frac{cx}{x+c}),$$ 
  which is equivalent to 
  $$\N_{k(\sqrt[n]{a})/k}(\mathbf{z}(x+c)^{m+1})=P'(\frac{cx}{x+c})(x+c)^{n(m+1)}.$$  
  The polynomial on the right hand side $P''(x)=P'(\frac{cx}{x+c})(x+c)^{n(m+1)}$ has irreducible factorisations over $k$ and over $K$ of the same form as $P'$, which give rise to an isomorphism between Brauer groups by Remark \ref{factorisation-determine-Br}.
  A further  change of variable $\mathbf{z}\mapsto\mathbf{z}(x+c)^{m+1}$ does not make any affect on the computation of unramified Brauer groups since all the generators are vertical, i.e. independent of $\mathbf{z}$. We may even assume the defining equation of $X''$ to be
  $$\N_{k(\sqrt[n]{a})/k}(\mathbf{z})=P''(x).$$
  The polynomial $P''$ has  irreducible factorisation  $P''=\prod_{i=0}^m P''_i$ over $k$. Each $P''_i$ factors over $K=k(\sqrt[n]{a})$ into $\frac{n}{n_i}$  irreducible polynomials of the same degree  $P''_i=\prod_{j=1}^{\frac{n}{n_i}}P''_{i,j}$. This factorisation information is the same as $P'$.
  
  To continue the  proof, we still use the notation $X'$, $P'_i$, $P'_{i,j}$ (instead of $X''$, $P''_i$, $P''_{i,j}$). Now we have additionally $\infty\in\mathcal{U}_{w_t}$ for $1\leq t \leq s$, in particular $X'_\infty(k_{w_t})\neq\varnothing$.  
  The fiber $X'_\infty$ is $k_v$-rational and has $k_v$-points for real places $v$ since  $\sqrt[n]{a}\in k_v$.
  \end{itemize}
  
  For each pair $(i,j)$, we fix a root $\theta_{i,j}\in\overline{k}$ of the irreducible polynomial $P'_{i,j}\in K[x]$. By Chebotarev's density theorem, we choose distinct non-archimedean places $v_{i,j}\in\Omega_k\setminus\{w_1,\ldots,w_s\}$ such that $X'_\infty(k_{v_{i,j}})\neq \varnothing$ and such that both  $x^n-a$ and $P'_i$ split completely in $k_{v_{i,j}}$.
  Then $k_{v_{i,j}}$ contains $\theta_{i,1},\theta_{i,2},\ldots,\theta_{i,\frac{n}{n_i}}$, and $\sqrt[n]{a}$.

  There are only finitely many places $w'$ such that $X'_\infty(k_{w'})=\varnothing$, they must be non-archimedean places and distinct from those $w_t$ and $v_{i,j}$ by construction.

  Proposition \ref{p1p1} ensures that there exists a polynomial $h\in k[x]$, which is viewed also as an endomorphism of $\PP^1$, such that 
    \begin{itemize}
      \item[\textup{(h1)}] For each $0\leq i\leq m$ and each $1\leq j\leq\frac{n}{n_i}$, the polynomial $h-\theta_{i,j}\in k_{v_{i,j}}[x]$ is irreducible,
      \item[\textup{(h2)}] For each $1\leq t\leq s$, the image $h(\PP^1(k_{w_t}))\subset\mathcal{U}_{w_t}$ and it meets  $\mathcal{U}_{w_t,\lambda}$ for each $\lambda\in \Lambda_t$,
      \item[\textup{(h3)}] For each place $w'$ such that $X'_\infty(k_{w'})=\varnothing $, the image $h(\PP^1(k_{w'}))$ meets $\pi'(X'(k_{w'}))$ which contains a certain non-empty open set of $\PP^1(k_{w'})$. 
  \end{itemize}

  Define  $\pi:X\to\PP^1$ to be the pull-back of $\pi':X'\to\PP^1$ along $h$. It is a projective model of the equation 
  $$\N_{k(\sqrt[n]{a})/k}(\mathbf{z})=P'(h(x)).$$
  Lemma \ref{irreducibility} (applied with $g=P'_i$, $\theta=\theta_{i,j}$, and $L=k_{v_{i,j}}$) together with (h1) implies that $P'_i\circ h$ is irreducible over $k$, and therefore separable as well. Then it follows from the separability of $P'$ that $P'\circ h$ is also separable.
  For $0\leq i\leq m$ and $1\leq j\leq\frac{n}{n_i}$,  Lemma \ref{irreducibility} (applied with $g=P'_{i,j}$, $\theta=\theta_{i,j}$, and $L=k_{v_{i,j}}$) together with (h1) again implies that $P'_{i,j}\circ h$ is irreducible over $K$.
   In summary, the separable polynomial $P=P'\circ h$ has irreducible factorisation  $P'\circ h=\prod_{i=0}^m P'_i\circ h$ over $k$.
   Each $P'_i\circ h$ factors as $P'_i\circ h=\prod_{j=1}^{\frac{n}{n_i}}P'_{i,j}\circ h$ over $K$ into $\frac{n}{n_i}$  irreducible polynomials of the same degree.
   
   Theorem \ref{cyclic construction} together with Remark \ref{factorisation-determine-Br}  allows us to conclude that the composition 
  $$B\To\br_\textup{nr}(X')\buildrel{h^*}\over\To\br_\textup{nr}(X)$$ 
  is a monomorphism with image generated by $\{\pi^*_\eta\left(\chi_a,P'_i\circ h\right)\mid 1\leq i\leq m\}$, each of which is the image by $h^*$ of $\pi'^*_\eta\left(\chi_a,P'_i\right)$, and it induces an isomorphism onto $\overline{\br}_\textup{nr}(X)$. 
  
    The degeneracy locus $Z=\Spec(k[x]/(P))$ of $\pi$ is exactly the pull-back $h^{-1}(Z')$ of the degeneracy locus of $\pi'$. The separability of $P$ indicates that $h:\PP^1\to\PP^1$ is \'etale over $Z'$. Then the fiber product $X=X'\times_{\PP^1,h}\PP^1$ is smooth (and projective). The unramified Brauer groups above are indeed $\Br(X')$ and $\br(X)$. This proves the conclusion (1) of the theorem.

  It follows immediately from the condition (h3) that $X$ has $k_v$-points for all $v$.
  The conclusion (3) is a consequence of the fact that $\im(\lambda'_{v})=\{0\}$ for any $v\in\Omega_k\setminus\{w_1,\ldots,w_s\}$ together with the natural composition  
  $$\lambda_v:X(k_v)\buildrel{h}\over\rightarrow X'(k_v)\buildrel{\lambda'_v}\over\rightarrow \widehat{B}.$$
  This locally constant map is determined by the image of $(X\setminus\pi^{-1}(Z))(k_v)$. As Brauer classes in $B$ are vertical, it suffices to consider the set of $x$-coordinates of these $k_v$-points, i.e. $\pi((X\setminus\pi^{-1}(Z))(k_v))$. Its image by $h$ is described by the condition (h2) when $v\in\{w_t\mid1\leq t\leq s\}$, which completes the proof of the conclusion (2). 
\end{proof}

\begin{rem} 
In the proof,  we start from  a very good model $X'$ given by \cite{VAV-model} of the equation $\N_{k(\sqrt[n]{a})/k}(\mathbf{z})=P'(x)$. We do not prove that the pull-back $X$ of $X'$ by $h$ is  a very good  model of $\N_{k(\sqrt[n]{a})/k}(\mathbf{z})=P'(h(x))$. In fact, the smooth projective variety $X$ is exactly the model given by \cite{VAV-model}, hence it is a very good model. But we are not going to present the details. 

\end{rem}

\bigskip

\bigskip

\tableofcontents

\bigskip

\bibliographystyle{alpha}
\bibliography{mybib1}

\end{document}